\documentclass[a4paper]{article}
\usepackage[utf8]{inputenc}
\usepackage[english]{babel}
\usepackage{amsmath}
\usepackage{amsfonts}
\usepackage{amsthm}
\usepackage{amssymb}
\usepackage{graphicx}
\usepackage{tikz}
\usetikzlibrary{shapes,snakes}
\usepackage{subcaption}
\usepackage{makecell}
\usepackage{booktabs}
\usepackage{multirow}

\usepackage{nicefrac}
\usepackage[short, nocomma]{optidef}

\usepackage[normalem]{ulem} 

\definecolor{MidnightBlue}{rgb}{0.1, 0.1, 0.44}
\definecolor{dorange}{rgb}{1,0.5,0}
\definecolor{darkgreen}{rgb}{0,0.5,0}
\definecolor{dkgreen}{RGB}{0,153,0}
\definecolor{violett}{RGB}{90,0,90}
\definecolor{turkis}{RGB}{0,147,168}
\definecolor{deepmagenta}{rgb}{0.8, 0.0, 0.8}
\definecolor{amber}{rgb}{1.0, 0.75, 0.0}
\definecolor{applegreen}{rgb}{0.55, 0.71, 0.0}
\definecolor{azure}{rgb}{0.0, 0.5, 1.0}

\newcommand{\seq}{\textup{Seq}}
\newcommand{\Int}{\textup{Int}}
\newcommand{\msp}{(MSP)}

\newcommand{\R}{\mathbb{R}}
\newcommand{\N}{\mathbb{N}}
\newcommand{\Z}{\mathbb{Z}}
\newcommand{\st}{\textup{s.t.}}
\newcommand{\PoS}{\textup{P\hspace{-0.3mm}o\hspace{-0.1mm}S}}

\newcommand{\cP}{\mathcal{T}}
\newcommand{\rep}{t}
\newcommand{\cL}{{\cal L}}

\newcommand{\cI}{\mathcal{I}}

\newcommand{\cE}{\mathcal{E}}
\newcommand{\cA}{\mathcal{A}}
\newcommand{\aA}{{\sf A}}

\newcommand{\cN}{\mathcal{N}}
\newcommand{\femax}{f_e^{\max}}
\newcommand{\femin}{f_e^{\min}}
\newcommand{\trans}{\textup{trans}}
\newcommand{\wait}{\textup{wait}}
\newcommand{\drive}{\textup{drive}}
\newcommand{\arr}{\textup{arr}}
\newcommand{\dep}{\textup{dep}}
\newcommand{\OD}{\textup{OD}}

\newcommand{\aux}{\textup{aux}}

\newcommand{\src}{\textup{source}}
\newcommand{\tar}{\textup{target}}

\newcommand{\length}{\textup{length}}
\newcommand{\depot}{\textup{depot}}
\newcommand{\first}{\texttt{first}}

\newcommand{\Cost}{\textup{cost}}
\newcommand{\cost}{4}
\newcommand{\linecost}{1}
\newcommand{\TT}{3}
\newcommand{\Start}{\alpha}
\newcommand{\End}{\omega}
\newcommand{\dur}{\delta}
\newcommand{\Veh}{\textup{Veh}}
\newcommand{\xbari}[1]{\bar{x}_1,\ldots,\bar{x}_{#1}}
\newcommand{\xbar}{(\bar{x}_1,\ldots,\bar{x}_{i-1})}
\newcommand{\fxbar}{f_i(\bar{x}_1, \ldots,\bar{x}_{i-1};x_i)}
\newcommand{\Fxbar}{F_i(\bar{x}_1, \ldots,\bar{x}_{i-1})}
\newcommand{\Part}[1]{\bar{x}^{#1}}
\newcommand{\PartOpt}[1]{\bar{x}^*_{#1}}

\newcommand{\timpass}{(TimPass)}
\newcommand{\lintimpass}{(LinTimPass)}

\newcommand{\timveh}{(TimVeh)}
\newcommand{\lintimveh}{(LinTimPassVeh)}

\theoremstyle{plain}
\newtheorem{theorem}{Theorem}
\newtheorem{lemma}[theorem]{Lemma}
\newtheorem{corollary}[theorem]{Corollary}
\theoremstyle{remark}

\theoremstyle{definition}
\newtheorem{definition}[theorem]{Definition}
\newtheorem{notation}[theorem]{Notation}

\newtheorem{example}[theorem]{Example}
\definecolor{darkgreen}{rgb}{0,0.5,0}

\title{Integrated Optimization of Sequential Processes: General Analysis and Application to Public Transport\footnote{This work was partially supported by DFG under SCHO 1140/8-1.}}
\author{Philine Schiewe$^1$\footnote{Corresponding author}, Anita Schöbel$^2$}

\date{}

\begin{document}
\maketitle

\vspace*{-2em}
\begin{center}
\small{
$^1$: Department of Mathematics,\\
Technische Universität Kaiserslautern,\\
Paul-Ehrlich-Straße~14, 67663 Kaiserslautern, Germany,\\
p.schiewe@mathematik.uni-kl.de

\medskip
$^2$: Department of Mathematics,\\
Technische Universität Kaiserslautern,\\
Gottlieb-Daimler-Straße~48, 67663 Kaiserslautern, Germany,\\
schoebel@mathematik.uni-kl.de\\
\smallskip
and\\
\smallskip
Fraunhofer Institute for Industrial Mathematics ITWM,\\
Fraunhofer Platz 1, 67663 Kaiserslautern, Germany
}
\end{center}
\bigskip

\begin{abstract}
  Planning in public transportation is traditionally done in a sequential process: After the network design process,
the lines and their frequencies are planned. When these are fixed, a timetable is determined and based on the
timetable, the vehicle and crew schedules are optimized. After each step, passenger routes are adapted
to model the behavior of the passengers as realistically as possible.
It has been mentioned in many publications that
such a sequential process is sub-optimal, and integrated approaches, mainly heuristics, are under consideration. 
Sequential planning is not only common in public transportation planning but also in many other applied
problems, among others in supply chain management, or in organizing hospitals efficiently. 
\medskip

The contribution of this paper hence is two-fold: on the one hand, we develop
an integrated integer programming formulation for the three planning stages line
planning, (periodic) timetabling, and vehicle scheduling which also includes the integrated optimization of the
passenger routes. This gives us an exact formulation rewriting the
sequential approach as an integrated problem. We discuss properties of the integrated formulation and apply
it experimentally to data sets from the {\sf LinTim} library. On small examples, we get an exact
optimal objective function value for the integrated formulation which can be compared with the outcome
of the sequential process.

On the other hand, we propose a mathematical formulation for general sequential processes which can be used
to build integrated formulations. For comparing sequential processes with their integrated counterparts
we analyze the \emph{price of sequentiality}, i.e., the ratio between the solution obtained by the
sequential process and an integrated solution. We also experiment with different possibilities for
partial integration of a subset of the sequential problems and again illustrate
our results using the case of public transportation. The obtained results may be useful for other
sequential processes.
\end{abstract}


\section{Introduction}

Public transport planning, as described in \cite{desaulniers2007public} encompasses various planning stages,
such as network design, line
planning, timetabling, passenger routing, vehicle scheduling and crew scheduling. Although these
stages are highly interdependent, they are usually solved sequentially with the passengers' paths
being adapted after each stage. Within the last few years, researchers noticed that this is a greedy-like
approach and started working on integrated approaches instead. Also planning problems in many other fields
are multi-stage optimization problems which are too complicated to be solved in an integrated manner and hence their different
stages are treated sequentially. Such sequential procedures are in particular common when infrastructural decisions are
involved. These are made first (finding locations, building offices, houses,
or shops, inventory decisions) while in a second phase the assignment, scheduling, or routing is done. An
integrated model would allow to make the infrastructural decision already in such a way that the second-stage
decision is as good as possible. This applies, e.g., for the design of supply chains, for production processes, or
for planning the operations in airports or hospitals. Also in these fields, researchers noticed that integrated optimization
may be more beneficial than only concentrating on the single stages. For example, 
integrated approaches have been compared to the sequential approach
in supply chain management, \cite{kiddeuro2018} or in production-distribution systems.
For the latter, \cite{Darvish18} analyzes differences between the sequential steps of location,
production, inventory, and distribution  decisions compared to an integrated approach while \cite{archetti2016inventory,absi2018comparing} focus on the benefits of integration for the production and inventory routing problem, respectively.

In the literature, there are also a few general ideas on how to integrate different planning stages or subproblems into
one optimization model for which general solution approaches can be formulated. One general scheme which allows to
find a local optimum is the eigenmodel proposed in \cite{Sch16}. Here, it is shown how to develop approaches
which iteratively solve subproblems being related to the different stages of a multi-stage problem. Convergence of
these iterative approaches is analyzed in \cite{JaeSch18}.
Formulating a sequential process as one integrated optimization model is further complicated if every
planning stage has its
own objective function, or if the whole problem is multi-objective by definition. In these cases, defining an
objective function for the integrated problem is not easy, taking all decision makers into account. It is also
not clear how the quality of results should be measured. For modeling situations through interwoven or complex systems, we
refer to \cite{Klamroth20,DIETZ2020581}. 
\medskip

In this paper we study sequential processes and how they can be integrated into one common optimization
problem. We are in particular interested how much we can improve the objective function value by solving
the integrated model instead of taking the sequential solution. 
We derive the notation, analysis and results for the general case of sequential processes but illustrate
everything in the case of public transport planning
where we focus on the three consecutive planning stages
line planning, (periodic) timetabling, and vehicle scheduling together with finding the passenger routes.
As a side-effect we present an integer programming
formulation for the integrated problem of line planning, timetabling, vehicle scheduling, and passenger routing.
\medskip

The remainder of the paper is structured as follows: In Section~\ref{sec-notation} we start by
providing a general scheme for dealing with sequential processes. We introduce the notation needed and
show how a sequential process can be transformed into an integrated formulation.
While this can be written down easily, 
formulating such an integrated problem in a practical case may be much harder. This is illustrated by
developing an integrated formulation for the sequential process of line planning, timetabling and vehicle
scheduling (including passenger routing) in Section~\ref{sec-LinTimvehPass}. This section includes
a literature review on work towards integrated optimization in public transportation as well
as a description of the sequential process and the integrated model for this case.
In Section~\ref{sec-PoS} we introduce the \emph{price of sequentiality (\PoS)} which is used to compare the
sequential solution to an integrated solution and we derive theoretical properties for the PoS. Section~\ref{sec:partially_integrated} is devoted to the case of partial integration and its price of sequentiality. Finally,
in Section~\ref{sec-experiments} we show experimental results using the {\sf LinTim}-library
\cite{lintimhp,lintim}. The paper is ended by a conclusion in Section~\ref{sec-conclusion}.

\section{Sequential processes and their integration}
\label{sec-notation}

Sequential processes and their integrated solution have been considered 
 in \cite{borndorfer2017passenger,mariediss,Sch16,kiddeuro2018,archetti2016inventory} for special applications.
 A more general framework which is the basis for the notation used in this paper
 has been introduced in \cite{philinediss}. Along these lines
 we start by defining a sequential process.
 
\begin{definition}\label{def:sequential}
  A \emph{sequential process} is given by a sequence of \emph{sequential problems}
  $(\seq_i \xbar)$, $i \in \{1, \ldots, n\}$, 
	\begin{align*}
			(\seq_i \xbar) \quad \min \; & \fxbar \\
			\st \quad  & x_i \in \Fxbar.
        \end{align*}
        For $i \in \{1,\ldots,n\}$ we call  $(\seq_i \xbar)$ \emph{stage $i$}
        of the sequential process. It includes $m_i \in \N$ variables
        $x_i \in \R^{m_i}$, a feasible set $\Fxbar \subseteq \R^{m_i}$ 
        and objective function $f_i \colon \R^{m_1} \times
          \ldots \times \R^{m_i} \to \R$.

          $f_i$ depends not only on the
          variables $x_i \in \R^{m_i}$ of the current stage but also on
          the variables of previous stages. When we want to indicate that
          only $x_i$ are variables, we separate them by a semicolon and
          write 
        $f_i(\xbari{i-1};\cdot) \colon \R^{m_i} \to \R$.
      \end{definition}
      
 Let $(\seq_1)$ be feasible with optimal solution $\bar{x}_1$ and let $(\seq_i(\bar{x}_1, \ldots, \bar{x}_{i-1}))$ be feasible with optimal solution $\bar{x}_i$ for each $i \in \{2, \ldots, n\}$. Then we call
 \[ \bar{x}=(\bar{x}_1,\ldots, \bar{x}_n) \]
 a \emph{sequential solution}.
\medskip

Let us assume that we are able to solve each of the sequential problems \linebreak
  $(\seq_i \xbar)$ for $i\in \{1, \ldots, n\}$ to optimality by known deterministic algorithms
  ${\aA}_1,\ldots,{\aA}_n$. Applying these algorithms ${\aA}_1,\ldots,{\aA}_n$
sequentially outputs a sequential solution and is hence called
\emph{sequential solution approach}. It is in short given by
\smallskip
    
\begin{center}
\small{
		\begin{tabular}{llll}
		\toprule
		\textbf{Problem} & \textbf{Algorithm} & \textbf{Input} & \textbf{Output}\\
		\midrule
        $(\seq_1)$ & ${\aA}_1$ & $\emptyset$ & $\bar{x}_1$\\
        $(\seq_2(\bar{x}_1))$ & ${\aA}_2$ & $\bar{x}_1$ & $\bar{x}_2$\\
        $(\seq_3(\bar{x}_1,\bar{x}_2))$ & ${\aA}_3$ & $\bar{x}_1,\bar{x}_2$ & $\bar{x}_3$\\
        & \vdots & \\
        $(\seq_n(\bar{x}_1,\bar{x}_2,\ldots,\bar{x}_{n-1}))$ & ${\aA}_n$ & $\bar{x}_1,\bar{x}_2,\ldots,\bar{x}_{n-1}$ & $\bar{x}_n$\\
        \bottomrule
      \end{tabular}  }   
\end{center}      
      \bigskip

Obviously, such a sequential process is a Greedy-like approach: in every stage
 we do the best possible. The question is if this is really the best possible
 for the sequential process as a whole.
 To give an answer we need to compare the sequential solution approach with the
\emph{integrated solution approach} of solving the following \emph{integrated problem}.

      \begin{definition}
        \label{def:integrated}
	For a sequential process as in Definition~\ref{def:sequential}, and weights $\lambda_i \geq 0$, $i \in \{1, \ldots, n\}$, the corresponding \emph{integrated optimization problem} is defined as a multi-stage optimization problem
	\begin{align*}
            \textup{\msp} \quad  \min\; & f(x_1,\ldots,x_n) =
                                        \sum_{i=1}^n \lambda_i \cdot f_i(x_1,\ldots, x_i)\\                            
			\st \quad x_1 & \in F_1\\
			x_2 & \in F_2(x_1)\\
			& \vdots \\
			x_n & \in F_n(x_1, \ldots, x_{n-1}).
		\end{align*}
          An optimal solution to \msp\ is called \emph{(optimal) integrated solution}.    
   	\end{definition}

The parameters $\lambda_1,\ldots,\lambda_n$ represent how important the objective of stage $i$
  is for the whole process. They can also be motivated by multi-objective optimization:
  Let us treat the multi-stage problem (MSP) in a multi-objective fashion.
  Instead of one common objective function, we then consider 
  \[f=\left(\begin{array}{c} f_1(x_1)\\ f_2(x_1,x_2) \\  \vdots \\ f_n(x_1,\ldots,x_n)\end{array}\right) \]
  as a \emph{vector-valued} objective function. It is known (see, e.g., \cite{Ehrgott05})
  that for $\lambda_i>0$ for all $i\in \{1,\ldots,n\}$ a (supported) Pareto solution will be found,
  and for $\lambda_i \geq 0$ for all $i\in \{1,\ldots,n\}$ a (supported) weakly Pareto solution is obtained.
\medskip

Clearly, the integrated solution is always better than a sequential solution.

\begin{lemma}
\label{Int-better-than-Seq}    
  Let $x^*=(x^*_1,\ldots,x^*_n)$ be an optimal solution to (MSP) and let $\bar{x}=(\bar{x}_1,\ldots,\bar{x}_n)$ be a sequential solution, i.e., $\bar{x}_i$ is an optimal solution to $(\seq_i(\bar{x}_1, \ldots, \bar{x}_{i-1}))$ for $i\in\{1,\ldots,n\}$. Then $f(x^*) \leq f(\bar{x})$.
\end{lemma}

\begin{proof} The result directly follows from the fact that the sequential solution is a feasible solution to (MSP). \end{proof}

Before we continue, we illustrate the notation in the example of public transport optimization.

\section{The sequential planning process in public transport and an integrated formulation}
\label{sec-LinTimvehPass}

In this section we develop an integrated formulation 
for the sequential process in public transport planning as depicted in Figure~\ref{fig:structure}.
The main effort for this lies in the formulation of the single planning stages as in Definition~\ref{def:sequential}.
These are provided in Section~\ref{sec:seq_sol}.
The integrated formulation is then put together in Section~\ref{sec:model}.

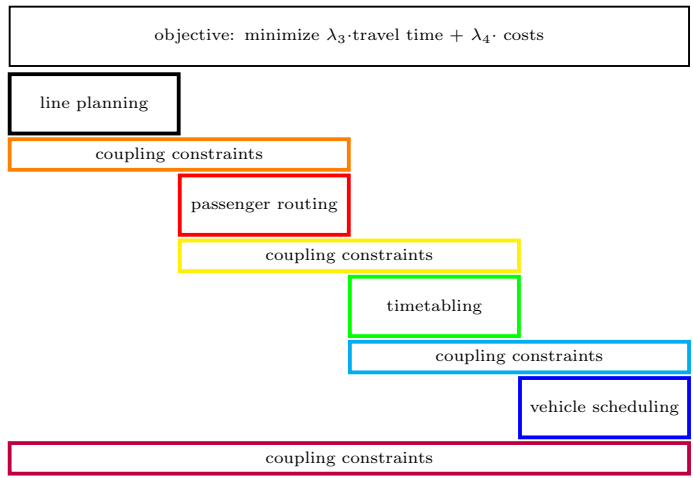
\begin{figure}	
	\begin{center}
	\resizebox{0.75\textwidth}{!}{	
	\begin{tikzpicture}[thick, font=\scriptsize]
	\node[draw, align=center, text width=9.75cm, minimum height=2.5em] (obj) at (-5,4) {objective: minimize  $\lambda_{\TT} \cdot$travel time + $\lambda_{\cost} \cdot$ costs};
	\node[draw, align=center, text width=2.25cm, minimum height=2.5em, ultra thick] (L) at (-8.75cm,3) {line planning};
	
	\node[draw=orange, align=center, text width=4.75cm, minimum height=1.25em, ultra thick] (LP) at (-7.5cm,2.25) {coupling constraints};
	
	\node[draw=red, align=center, text width=2.25cm, minimum height=2.5em, ultra thick] (P) at (-6.25cm,1.5) {passenger routing};
	
	\node[draw=yellow, align=center, text width=4.75cm, minimum height=1.25em, ultra thick] (PT) at (-5cm,0.75) {coupling constraints};
	
	\node[draw=green, align=center, text width=2.25cm, minimum height=2.5em, ultra thick] (T) at (-3.75cm,0) {timetabling};
	
	\node[draw=cyan, align=center, text width=4.75cm, minimum height=1.25em, ultra thick] (TV) at (-2.5cm,-0.75) {coupling constraints};
	
	\node[draw=blue, align=center, text width=2.25cm, minimum height=2.5em, ultra thick] (V) at (-1.25cm,-1.5) {vehicle scheduling};
	
	\node[draw=purple, align=center, text width=9.75cm, minimum height=1.25em, ultra thick] (all) at (-5cm,-2.25) {coupling constraints};
	\end{tikzpicture}}
	\caption{Structure of the integrated line planning, timetabling and vehicle scheduling problem with passenger routing.} \label{fig:structure}
\end{center}
\end{figure}

\subsection{Literature review:  planning in public transportation}
\label{sec-literature}

We start by providing a short review on literature for planning in public transportation.
The established sequential solution approach consists of network design, line planning,
timetabling, vehicle scheduling and crew scheduling (see, e.g., \cite{CedWil86,bussieck1997discrete,
desaulniers2007public,guihaire2010transit,LLB18}).
Note that other sequential approaches are also possible and under research. In \cite{PSSS17,bouman_et_al:OASIcs:2020:13142}, a vehicle schedule is computed before the timetable while in \cite{MicSch07}, a vehicle schedule is constructed first, followed by lines and a timetable. For a theoretical analysis of the different sequential solution approaches, see \cite{Sch16}. Here, we concentrate on line planning, passenger routing, periodic timetabling and vehicle scheduling
as the core stages. An overview on line planning is provided in \cite{Sch10b} and on (periodic) timetabling
in \cite{Nac98,Lieb06}. For vehicle scheduling we refer to \cite{bunte2009overview}. Research still carries on for each of these stages, see \cite{FHSS-CASPT18,SAHIN2020102726} for
extensions of line planning, \cite{BornLindRoth19,galli2018modern,Borndoerfer20} for new features and
procedures in periodic timetabling and \cite{liebchen18,Schmidt20,Guedes20} for work on vehicle
scheduling. We also refer to the references therein.
Besides research on robustness in public transport and work on case studies, integration of the mentioned planning stages is a current topic and the focus of this paper.
There exists the following
literature for integrating subsets of the stages; mostly in a heuristic manner. 

\paragraph{Integration of line planning and timetabling}
In \cite{liebchen2008linien} line segments are connected to lines during the timetabling stage in a
heuristic approach. 
Iterative approaches of alternately re-optimizing line plans and timetables are presented in \cite{burggraeve2017integrating,Goverde19}.
Exact integer programming models are presented in \cite{Rittner2009,kaspi2013service,Puerto20} and solved by a column generation approach, a cross entropy heuristic and a matheuristic, respectively.

\paragraph{Integration of timetabling and vehicle scheduling}
While we are interested in periodic timetabling,
most models for integrated timetabling and vehicle scheduling deal with \emph{aperiodic} timetables and solve
the resulting integrated problem heuristically. The timetable and the vehicle schedule are often optimized
iteratively, see \cite{guihaire2010transit, petersen2013simultaneous, schmid2015integrated, fonseca2018matheuristic}
but also with matheuristics in \cite{Carosi19} or by a bi-level model in \cite{yue2017integrated}.
A periodic version for integrating timetabling and vehicle scheduling can be found in \cite{van2021integrated} while in \cite{lindner2000train} periodic vehicle scheduling is
integrated indirectly by using an approximation of the costs as objective.
In \cite{narayan} periodic vehicle scheduling constraints are added to a periodic event scheduling
problem. Pareto solutions for integrating periodic timetabling with aperiodic
vehicle scheduling are identified in \cite{Heureka21}.

\paragraph{Integration of line planning, timetabling, and vehicle scheduling}
Integrating the three steps usually takes two objectives, namely travel time for the passengers and
costs of the public transport supply, into account. 
In \cite{liebchen2008linien,li2018metro}, simplified integrated models are presented by adding further constraints to the periodic and aperiodic timetabling problem, respectively. 
Heuristic approaches which are based on changing the order of the sequential problems are given in \cite{MicSch07,PSSS17} and generalized in \cite{Sch16}. Finally, \cite{PaeSchiSch18} only focuses on the costs of the
public transport supply with the minimal restriction that all passengers can travel.

\paragraph{Integration of passenger routing}
Line planning has been integrated with passenger routing for nearly 15 years, starting with \cite{ScSc06a}
including penalties for transfers. An efficient solution approach is provided by \cite{BGP07}, still, the
integrated problem can only be solved for small to medium sized instances.
In a capacitated setting it is especially difficult to guarantee that all passengers can travel on shortest paths, see \cite{goerigk2017line} for integer programming formulations and \cite{GSS14} for game theoretic approaches. 
The integration of timetabling
with passenger routing came up later, but has been researched intensively since this time, see \cite{Atmos2010-schmidt,siebert2013experimental,mariediss,SchmSch14,robenek2017hybrid,borndorfer2017passenger,SchiSch18,polinder2021timetabling}. These papers provide integer programming
models, an analysis and iterative solution approaches, while 
\cite{GGNS} uses a SAT formulation and a SAT-based solution approach.
\bigskip

The integration with network design is also a subject of current
research, while the integration of vehicle and crew scheduling is
advanced, see, e.g., \cite{HDD01} and already found entrance in
current professional planning tools.

\subsection{The sequential process for public transport optimization}\label{sec:seq_sol}
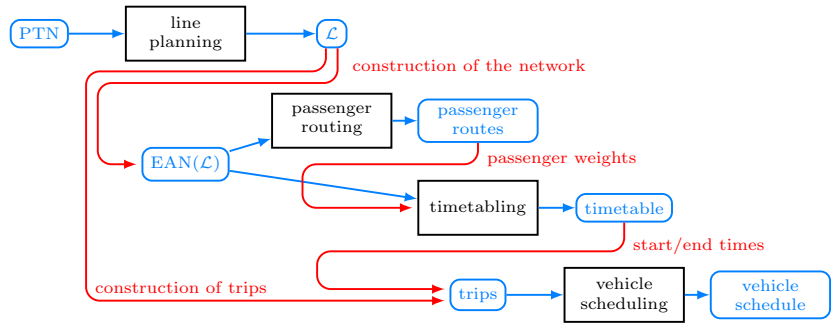
\begin{figure}	
	\begin{center}
	\resizebox{0.9\textwidth}{!}{
	\begin{tikzpicture}[thick, font=\scriptsize, scale=0.85]
	\node[draw, align=center, rounded corners, azure] (PTN) at (-11.25,3) {PTN};	
	
	\node[draw, align=center, text width=1.5cm, minimum height=2.25em] (L) at (-8.75cm,3) {line planning};
	
	\node[draw, align=center, rounded corners, azure] (cL) at (-6.25,3) {$\cL$};	
	
	\node[draw, align=center, rounded corners, azure] (EAN) at (-8.75,0.75) {EAN$(\cL)$};	
	
	\node[draw, align=center, text width=1.5cm, minimum height=2.25em] (P) at (-6.25cm,1.5) {passenger routing};
	
	\node[draw, align=center, rounded corners, azure, text width=1.5cm] (R) at (-3.75,1.5) {passenger routes};

	\node[draw, align=center, text width=1.5cm, minimum height=2.25em] (T) at (-3.75cm,0) {timetabling};
	
	\node[draw, align=center, rounded corners, azure] (Tim) at (-1.25,0) {timetable};	
	
	\node[draw, align=center, rounded corners, azure] (trips) at (-3.75,-1.5) {trips};			
	\node[draw, align=center, text width=1.5cm, minimum height=2.25em] (V) at (-1.25cm,-1.5) {vehicle scheduling};
	
	\node[draw, align=center, rounded corners, azure, text width=1.5cm] (VS) at (1.25,-1.5) {vehicle schedule};	
	
	\draw[-latex, azure] (PTN) to (L);
	\draw[-latex, azure] (L) to (cL);
	\draw[-latex, azure] (EAN) to (P);
	\draw[-latex, azure] (EAN) to (T);
	\draw[-latex, azure] (P) to (R);
	\draw[-latex, azure] (T) to (Tim);
	\draw[-latex, azure] (trips) to (V);
	\draw[-latex, azure] (V) to (VS);
	
	  \draw[-latex,rounded corners=0.2cm,shorten >=2pt, red] ([shift={(0.1,0)}]cL.south) -- (-6.15,2.15)--(-10.25,2.15)--(-10.25,0.75)-- (EAN);
	  \draw[-latex,rounded corners=0.2cm,shorten >=2pt, red] ([shift={(-0.1,0)}]cL.south) -- (-6.35,2.35)--(-10.45,2.35)--(-10.45,-1.6)-- ([shift={(0,-0.1)}]trips.west);
	  \draw[-latex,rounded corners=0.2cm,shorten >=2pt, red] (R) -- (-3.75,0.75)--(-6.75,0.75)--(-6.75cm,0)-- (T);
	  \draw[-latex,rounded corners=0.2cm,shorten >=2pt, red] (Tim) -- (-1.25,-0.75)--(-6.5,-0.75)--(-6.5,-1.4)-- ([shift={(0,0.1)}]trips.west);
	  
	  \node[red, align=left, anchor=west] at (-6.05,2.45) {construction of the network};
	  \node[red, align=left, anchor=west] at (-3.75,0.85) {passenger weights};
	  \node[red, align=left, anchor=west] at (-1.25,0-.65) {start/end times};
	  \node[red, align=left, anchor=west] at (-10.45,-1.4) {construction of trips};
	\end{tikzpicture}}
	\caption{
Dependencies of the sequential problems line planning, passenger routing, timetabling and vehicle scheduling problem.} \label{fig:sequential_dependence}
\end{center}
\end{figure}

In this section we consider our sequential stages in public transport optimization, namely
  line planning, passenger routing, timetabling, and
  vehicle scheduling. The goal is to formulate these stages as a sequential process in order to
  receive the corresponding multi-stage problem (MSP).
  \medskip
  
For each of the single stages in public transport
optimization, integer programming formulations are known. On a first glance, it seems to be
easy to put these together to an integrated formulation (MSP). However, for doing so the
integer programs of the single stages have to be reformulated such that the decisions made in
previous stages appear as parameters in subsequent stages.
This can be illustrated when looking at Figure~\ref{fig:sequential_dependence}:
For example, the first stage,
line planning, outputs a set of lines ${\cal L}$. This set of lines is used to build
the event-activity network (EAN) which is the input for the second stage (passenger
routing) and for the third stage (timetabling). The classic formulations for timetabling and passenger
routing assume that the actual event-activity network is given and fixed. 
For developing an integrated formulation we do not construct the event-activity network explicitly
but reformulate the passenger routing and timetabling problem such that the decision variables from
the line planning stage model implicitly make sure that the correct event-activity problem is used.
\medskip

An overview on the notation 
of the following section is summarized in Table~\ref{tab:param:all}.
For denoting in which planning stages the parameters are relevant, we abbreviate line planning as Lin, passenger routing as Pass, timetabling as Tim and vehicle scheduling as Veh.  The parameters are described in the respective sections. 

\begin{table}[h]
{\small    
\begin{center}
\begin{tabular}{p{1.6cm}p{7.8cm}p{1cm}}
\toprule
\textbf{Notation} & \textbf{Explanation} & \textbf{Stage }\\ 
\midrule
$(V,E)$ & public transport network, $V$ stops, $E$ direct connections &  \\
\midrule
$\femin, \femax$ & lower/upper frequency bounds on edge $e\in E$ & Lin \\
$\cL^0$ & line pool, set of lines to chose from & Lin \\ 
$\Cost_l$ & cost of line $l$ & Lin  \\
\midrule
$C_{u,v}$ & passenger demand between stop $u \in V$ and $v \in V$ & Pass \\ 
\midrule
$L_e^{\drive}, U_e^{\drive}$ & lower/upper bound on the travel time on edge $e \in E$ &  Tim\\ 
$L_v^{\wait}, U_v^{\wait}$ & lower/upper bound on the dwell time in stop $v \in V$ & Tim \\
$L_v^{\trans}, U_v^{\trans}$ & lower/upper bound on the transfer time in stop $v \in V$ & Tim \\
$T$ & length of the planning period & Tim \\ 
\midrule
$\cP$ & set of periods considered for vehicle scheduling & Veh \\ 
\bottomrule
\end{tabular}
\end{center}
\caption{Parameters/data needed for the public transport problems.} \label{tab:param:all}
}
\end{table}

Let a public transportation network (PTN) $(V,E)$ be given consisting of stops $V$ and direct
connections (e.g.~tracks) $E$ between the nodes. Furthermore, let $C_{u,v}$ be the passenger demand 
for each origin-destination pair $(u,v) \in\OD \subseteq V \times V$.
The goal is to find a line plan, a timetable, passenger routes
  and a vehicle schedule such that a weighted sum of
  the costs and the travel time
  of the passengers is as small as possible. As travel time we focus on
  the \emph{perceived} travel time which includes a penalty for every
  transfer. The costs consist of distance-based costs and time-based costs and the number of vehicles needed for operating the vehicle schedule.

\subsubsection{Stage 1: Line planning (Lin)}\label{sec:seq:line_planning}
Line planning is the first stage of the planning process in public transportation
considered here. For the first stage, no reformulation is necessary,
we can just take any of the known formulations.
Here, we use binary variables $y_l$, $l \in \cL^0$, determining which lines from a given line
pool are operated. $(\seq_1)$ is given as:
        
\begin{alignat}{5}
		(\seq_1) \quad \min \; f_1(y) && :=\sum_{l \in \cL^0} y_l \cdot \Cost_l	&&& \nonumber \\ 	
	 	 \st && \sum_{\substack{l \in \cL^0 \colon e \in l}} y_l & \geq  \femin &&\qquad e \in E \label{L1} \\
	 	&& \sum_{\substack{l \in \cL^0 \colon e \in l}} y_l & \leq  \femax && \qquad e \in E \label{L2}\\
	 	&& y_l & \in \{0,1\} && \qquad l \in \cL^0 \nonumber
\end{alignat}   
with $F_1:=\{y \in \R^{|\cL^0|}: \eqref{L1},\eqref{L2} \mbox{ are satisfied}\}$.

The lower and upper edge frequency bounds $\femin, \femax$ for every edge $ e \in E$
make sure that all passengers can be transported and that the capacity on the edges is satisfied.
The objective minimizes the costs for the infrastructure provider where
$\Cost_l$ approximate the cost for operating line $l \in \cL^0$.

\subsubsection{Stage 2: Passenger routing (Pass)} 
Passenger routing is a shortest path problem in the so-called
  \emph{event-activity network (EAN)} $(\cE,\cA)$. It consists of
  \begin{itemize}
    \item events $\cE$ which represent arrivals and departures of lines at stations,
    \item activities $\cA$ representing drive activities $\cA_{\drive}$
      between stations, wait activities $\cA_{\wait}$ at stations and transfer activities $\cA_{\trans}$ which
      allow passengers to change between lines.
    \end{itemize}
    
For modeling passenger routing, we also need origin and destination events. If these are included, we call the extended event-activity network $(\bar{\cE},\bar{\cA})$.   
For a formal definition, see Notations~\ref{nota-EAN-formal} and \ref{nota-EAN-routing} in \ref{app-EAN}. We assume that lower bounds $L_a$ on the activity durations can be derived from the PTN (see \eqref{explain-bounds} in the appendix
and Table~\ref{tab:param:all})
and are interested in a shortest path w.r.t.\ these weights $L_a$.
\medskip

Finding shortest paths in a given network is a well-known problem. Here, we use the
classic flow formulation with variables $p_a^{u,v}$ for determining whether activity
$a \in \cA$ is part of a
shortest path from $u$ to $v$. The flow conservation constraints are typically written
as $A p^{u,v}=b^{u,v}$ where $A$ is the node-arc-incidence matrix
of the underlying network (the complete formulation is provided in~\ref{sec:app:passenger_routing}).
\medskip

In our case, the network depends on the lines chosen in stage~1. As underlying network
we nevertheless take the event-activity network which contains all potential lines from the line pool ${\cal L}_0$
but use the first-stage variables $y$ as parameters for stage~2 to make implicitly sure that the
correct network is considered. Let us collect the wait and drive activities
of line $l$ in the set $\cA(l)$. Constraints
\[ p_a^{u,v} \leq  \bar{y}_l \quad  (u,v) \in \OD, a \in \cA(l) \]
then guarantee that passengers of OD pair $(u,v)$ can only be routed on activities $a$
which belong to a line $l$ chosen in stage~1. The formulation of the second-stage problem hence
becomes

\begin{alignat}{7}
&&(\seq_2(\bar{y})) \quad \min f_2(\bar{y};p)&&&:=\sum_{(u,v) \in \OD} && C_{u,v} \cdot \sum_{a \in \cA}  && p_a^{u,v}  \cdot L_a  &&\nonumber \\
&&\st \quad  & &&A \cdot ({p}^{u,v}) && =  b^{u,v} &&\quad (u,v) \in \OD \label{P1}\\
& && &&{p}_a^{u,v} && \leq  \bar{y}_l &&\quad  (u,v) \in \OD, a \in \cA(l) \label{eq-LP1}\\
& && &&{p}_a^{u,v} && \in \{0,1\} && \quad  (u,v) \in \OD, a \in \bar{\cA} \nonumber
\end{alignat}
with $F_2(\bar{y}):=\{p_a^{u,v} \in \{0,1\}: a \in \bar{\cA}, (u,v) \in \OD, (\ref{P1}),(\ref{eq-LP1}) \mbox{ are satisfied}\}$.
\medskip

$A$ is the node-arc incidence matrix of the event-activity network with respect to the
line pool ${\cL^0}$ and $b^{u,v}$ contains the demand vector of the flow conservation
constraints, i.e., constraint \eqref{P1} ensures correct passenger flows.
The objective minimizes the sum of lower bounds $L_a$ of the travel times over
all drive and wait activities as approximation of the travel times. The travel time can be generalized to the perceived travel time by adding a penalty for each transfer. We do not add these transfer penalties here to not further complicate the notation.
We do not consider vehicle capacities for routing passengers but assume that the capacity is always sufficient such that passengers can travel on a shortest path. For settings with high demand it would be necessary to include capacity constraints into the routing process as in \cite{burggraeve2017integrating,goerigk2017line,PaeSchiSch18,Goverde19}.

\subsubsection{Stage 3: Timetabling (Tim)}\label{sec:seq:timetabling} 
We consider \emph{periodic} timetabling, i.e., the timetable is repeated every $T$ minutes,
e.g., every hour. The underlying model is the \emph{periodic event scheduling problem (PESP)}
in the event-activity network which we already used in stage~2 for determining the passenger routes.
For details we refer again to Notation~\ref{nota-EAN-formal} in \ref{app-EAN}.
The integer programming model for periodic timetabling uses
variables $\pi_i$ which contain the arrival/departure time of event $i$ for each event
in the EAN and so-called modulo parameters $z_a \in \Z$ for all $a \in \cA$ to model
the periodicity of the timetable. The basic PESP constraint is
\begin{equation}
  \label{tt-basic}
  L_a \leq \pi_j - \pi_i + z_a \cdot T  \leq  U_a \mbox{ for all } a=(i,j) \in \cA
\end{equation}  
where besides the lower bounds $L_a$ from passenger routing we also use bounds upper bounds
$U_a$ and $L_a$ on the duration of activity $a$. (Both can be extracted from the PTN, see
again \eqref{explain-bounds}.)
The objective is to minimize the sum of travel times along the activities. We get
\begin{equation}
  \label{tt-objective}
  \sum_{a \in \cA}  w_a \cdot (\pi_j - \pi_i + z_a  \cdot T) 
\end{equation}

where the given weights $w_a$ represent the number of passengers using activity $a$,
see \eqref{PESP} in the appendix. There are dependencies to stage~1 and to stage~2:
\begin{itemize}
\item First, a timetable is only needed for
events which exist, i.e., which belong to lines which have been chosen in stage~1.
In order to make this dependency explicit we introduce binary variables
$\eta_a=y_{l_1} \cdot y_{l_2}$ 
for every activity $a=(i,j)\in \cA(l_1,l_2)$, i.e.,
with event $i$ being in line $l_1$ and event $j$ in line $l_2$.
Hence, $\eta_a=1$ if and only if both endpoints of activity $a$ belong to lines which have been selected in
stage~1. Adding $\eta_a$ to \eqref{tt-basic} results in
\begin{equation}
  \label{tt-new}
  L_a \cdot \eta_a  \leq \pi_j - \pi_i + z_a \cdot T  \leq  U_a + M \cdot (1-\eta_a)
  \mbox{ for all } a=(i,j) \in \cA
\end{equation}  
and makes (for $M$ sufficiently large) sure that the timetable constraints need only be satisfied if both of the
considered events belong to lines that have been selected in stage~1.
\item Second, the weights $w_a$ in the objective function depend on the
  passenger routes of stage~2. Using the variables $p_a^{u,v}$ of stage~2,
  we can compute them by 
  \[
    w_a= \sum_{\substack{(u,v) \in \OD:\\ p_a^{u,v}=1}} C_{u,v}.
  \]
\end{itemize}
In the objective function of the third-stage problem we hence plug in the definition of $w_a$
in \eqref{tt-objective} and replace the constraints \eqref{tt-basic} by \eqref{tt-new}.
We receive:
\begin{alignat}{5}
&& (\seq_3(\bar{y},\bar{p})) \ \ \min f_3(\bar{y},\bar{p};\pi,z) &:=  \sum_{a \in \cA} \sum_{(u,v) \in \OD}  \bar{p}_a^{u,v}\cdot   && C_{u,v} \cdot (\pi_j - \pi_i +  z_a  \cdot T) \nonumber \\
&&\mbox{s.t.} \ \pi_j - \pi_i + z_a \cdot T & \geq  \eta_a \cdot L_a \ \  && a=(i,j) \in \cA \label{TT1}\\
&&\pi_j - \pi_i + z_a \cdot T & \leq   U_a + M \cdot (1-\eta_a) \ \  && a=(i,j) \in \cA \label{TT2}\\
&& \eta_a & =  \bar{y}_{l_1} \cdot \bar{y}_{l_2}  && a \in \cA(l_1,l_2) \ \ \label{TT3}\\
&& \pi_i & \in  \{0, \ldots, T-1\} \quad  && i \in \cE \nonumber \\
&&z_a & \in  \Z \ && a \in \cA \nonumber \\
&&\eta_a & \in  \{0,1\} \ && a \in \cA \nonumber
\end{alignat}
with
\begin{eqnarray*}
  F_3(\bar{y},\bar{p})&:= &\{\pi \in \{0,\ldots,T-1\}^{|\cE|}, z \in \{0,1\}^{|\cA|} \colon \exists \eta \in \{0,1\}^{\cA} \colon \\
                      && (\ref{TT1}),(\ref{TT2}),(\ref{TT3}) \mbox{ are satisfied}\}.
\end{eqnarray*}
                         
Note that \eqref{TT3} can be easily linearized (see \eqref{tmp1}-\eqref{tmp3} in the appendix),
and that the auxiliary variables $\eta_a$ are not explicitly mentioned as variables of $F_3$.

\subsubsection{Stage 4: Vehicle scheduling (Veh)}
From the lines and the timetable we can determine \emph{trips}: These are journeys from
the start station of a line to its end station which have to be operated at the times
which have been fixed in the timetable. To obtain the trips we need to roll-out the timetable for a given set of periods $\cP$.
The goal of vehicle scheduling is to assign vehicles to these trips
such that all trips are operated with minimal costs.
The cost function includes costs for the vehicles needed,
for the (empty) distance driven, and the time the vehicles are used.
Vehicle scheduling can be modeled as a flow problem using variables $x_{\tau_1,\tau_2}$
which are 1 if trips $\tau_1, \tau_2$ are operated by the same vehicle directly after each other.
The formulation can be found in \eqref{VEHSCHED} in \ref{sec:app:veh}.
\medskip

Also here, we have dependencies to former stages, since the trips depend on the lines
determined in stage~1 and on the timetable determined in stage~3, see Figure~\ref{fig:structure}.
Similar as before,
these dependencies can be made explicit by using the line-choice
variables $y$ of stage~1 and the timetable $\pi$ of stage~3. The resulting IP formulation
needs further auxiliary variables, namely the start and end times $\Start_{\tau}$ and $\End_{\tau}$
and the duration $\dur_{\tau}$ of the trips. These variables can be directly determined from the variables of the
former stages and are hence not needed explicitly. The resulting IP formulation
\eqref{TV1} - \eqref{var2} is given in \ref{sec:app:veh}.
In abstract form it reads as
\begin{eqnarray*}
\min \quad && f_4(\bar{y}, \bar{p}, \bar{\pi}, \bar{z}; x)\\ 
&& \st  \quad x \in F_4(\bar{y},\bar{p},\bar{\pi},\bar{z})
\end{eqnarray*}
where $F_4(\bar{y},\bar{p},\bar{\pi},\bar{z})$ describes the set of feasible vehicle schedules
with respect to the decisions made in the former stages.

\subsection{An integrated model for line planning, timetabling, vehicle scheduling and
  passenger routing}
\label{sec:model}

We finally can put the sequential formulations of Section~\ref{sec:seq_sol}
together and receive an integer program for \msp:

\begin{align*}
	\min \; \lambda_\TT \cdot f_3(y,p,\pi,z) & + \lambda_\cost \cdot f_4(y,p,\pi,z,x)\\
	\st \quad y & \in F_1\\
	p & \in F_2(y)\\
	(\pi,z) & \in F_3(y,p)\\
	x & \in F_4(y,p,\pi,z)
\end{align*}

The IP formulation is depicted in a schematic way in Figure~\ref{fig:structure}.
Note that the objectives $f_1$ and $f_2$ are only auxiliary objectives, as the line costs $f_1$ are a first approximation of the costs $f_4$ and the travel time $f_2$ according to the lower bounds is an approximation of the travel time $f_3$. We therefore only evaluate the objectives $f_3$ and $f_4$ in the integrated problem.

\section{The price of sequentiality}
\label{sec-PoS}
 
In order to analyze the differences of the sequential and the integrated solution,
we use the \emph{price of sequentiality} based on \cite{philinediss}. It quantifies
how well \msp{} is approximated by the sequential solution approach.
Similar ideas can be found in \cite{mariediss},
in \cite{borndorfer2017passenger} for integrating passenger routing and timetabling and
in \cite{kiddeuro2018,archetti2016inventory} for supply chain management.

\begin{definition}
	Let $(x_1^*,\ldots, x_n^*)$ be optimal for \msp{} with $f(x_1^*, \ldots, x_n^*) >0$ and let $\bar{x}_i$ be optimal for $(\seq_i(\bar{x}_1,\ldots,\bar{x}_{i-1}))$, $i \in \{1, \ldots, n\}$. Then the \emph{price of sequentiality} is defined as
	\[PoS = \frac{f(\bar{x}_1, \ldots, \bar{x}_n) - f(x_1^*, \ldots, x_n^*)}{f(x_1^*, \ldots, x_n^*)}.\] 
\end{definition}

Note that the sequential solution and therefore the price of sequentiality depends on the algorithms  ${\aA}_1,\ldots,{\aA}_n$ used in the sequential solution approach.
Due to Lemma~\ref{Int-better-than-Seq}, the price of sequentiality is always positive.
For the remainder of the paper, we assume that the optimal objective value of the integrated problem
satisfies $f(x^*)>0$.

Note that $PoS=0$ means that an optimal solution of the integrated problem can be found by the sequential solution approach.  
As the sequential approach is a Greedy-like approach that carries out its
major steps in a predefined order, 
the definition of the sequential problems including their objective functions and the order in which they
are solved can be very important.
If such a greedy algorithm is known to be
optimal, the sequential solution approach
hence finds an optimal solution when the order is chosen accordingly. 
This is the case for special problems, e.g., for the matroid problem (see \cite{philinediss} for details defining a sequential process in the case of matroid optimization), but also for shortest paths in acyclic graphs or if the sequential problems are independent, i.e., when for all $i$ the variables $x_i$ are only part of feasibility constraints $F_i$ and objective functions $f_i$.

In general, the price of sequentiality can be unbounded even for two linear programs with $\lambda = (1,1)^t$ as shown in the following example.

\begin{example}\label{ex:pos_unbounded}
Consider the sequential problems
\begin{align*}
	(\seq_1) \quad &  \{\min x_1 \colon x_1 \leq 1, x_1 \geq 0\}\\
	(\seq_2(\bar{x}_1)) \quad & \{\min  x_2 \colon x_2 \leq 1 -\bar{x}_1, x_2 \geq 1 - N \cdot \bar{x}_1,  x_2 \geq 0\}.
\end{align*}
For $N\geq 0$, the optimal objective value of the sequential solution approach is $1$. However, for the integrated problem with $\lambda = (1,1)^t$, i.e., for
\[\msp \quad \{\min x_1+x_2 \colon x_1 +x_2 \leq 1, N\cdot x_1 + x_2 \geq 1, x_1, x_2 \leq 1, x_1, x_2 \geq 0\}\]
the optimal objective value is $\frac{1}{N}$. Thus, the price of sequentiality satisfies
\[PoS = \frac{1- \frac{1}{N}}{\frac{1}{N}} = N-1 \overset{N \to \infty}{\to} \infty.\]
\end{example}

For practical applications where the integrated problem cannot be solved in reasonable time, it is interesting to
  know if \PoS\ is bounded. Often it is much easier to compute the price of sequentiality for each of the objective functions $f_i$ separately. We can then use these bounds to construct a bound on the PoS for any parameter $\lambda=(\lambda_1, \ldots, \lambda_n)$ with $\lambda_i  \geq 0$ for all $i \in \{1, \ldots, n\}$.

\begin{theorem}\label{thm:bound-pos-single-objective}
	Let $(\seq_i(\xbari{i-1}))$, $i \in \{1, \ldots, n\}$, be a family of sequential problems with sequential solution $\bar{x}$ and let $x^*$ be an optimal solution of the corresponding multi-stage problem $\msp$ for $\lambda_i \geq 0$, $i \in \{1, \ldots, n\}$. Let $x^i$ be an optimal solution of $\msp$ with objective function $f_i$, i.e., for $\lambda_i'=\lambda_i$, $\lambda_j'=0$ for all $j \neq i$ and let $f_i(x^i)>0$. Let
	\[\PoS^i := \frac{f_i(\xbari{i})-f_i(x^i_1, \ldots x_i^i)}{f_i(x^i_1, \ldots x_i^i)} \quad \mbox{ for } i \in \{1, \ldots, n\}.\]
	Then 
	\[\PoS = \frac{f(\bar{x})-f(x^*)}{f(x^*)} \leq 
	\max_{\substack{i \in \{1, \ldots, n\}:\\ \lambda_i>0}} \PoS^i.\]
In particular, $\PoS$ is bounded if all $\PoS^i$ are.
\end{theorem}

\begin{proof}
	Let $\cI:=\{i \in \{1, \ldots, n\} \colon \lambda_i > 0\}$ be the index set of subproblems with $\lambda_i>0$. As $f(x^*)>0$ by assumption, there is a $k \in \{1, \ldots, n\}$ with  $\lambda_k>0$, i.e., $\cI \neq \emptyset$.
	We know
	\begin{align*}
		f(\bar{x}) &= \sum_{i=1}^n \lambda_i \cdot f_i(\xbari{i})  = 
		\sum_{i \in \cI}\lambda_i \cdot f_i(\xbari{i}) \\
		f(x^*) &= \sum_{i=1}^n \lambda_i \cdot f_i(x^*_1, \ldots, x^*_i)\\ 
		& \geq \sum_{i=1}^n \lambda_i \cdot f_i(x^i_1, \ldots, x^i_i)\\ 
		& =  \sum_{i\in \cI} \lambda_i \cdot f_i(x^i_1, \ldots, x^i_i)>0.
	\end{align*}
	Thus 
	\begin{align*}
		\frac{f(\bar{x})-f(x^*)}{f(x^*)} 
		&\leq \frac{\sum_{0i \in \cI} \lambda_i \cdot f_i(\xbari{i})
					- \sum_{i \in \cI} \lambda_i \cdot f_i(x^i_1, \ldots x_i^i)}
					{\sum_{i \in \cI} \lambda_i \cdot f_i(x^i_1, \ldots x_i^i)}\\
		& = \frac{\sum_{i \in \cI} \lambda_i \cdot f_i(x^i_1, \ldots x_i^i) \cdot \PoS^i}
					{\sum_{i \in \cI} \lambda_i \cdot f_i(x^i_1, \ldots x_i^i)}\\
		& \leq \max_{i \in \cI} \PoS^i.
	\end{align*}
\end{proof}

We turn to our example in public transport. Here, for the integrated line planning, passenger routing, timetabling and vehicle scheduling problem, the price of sequentiality is bounded for every feasible solution
\begin{itemize}
	\item when only the passengers' travel time $f_\TT$ is considered and 
	\item when only the number of vehicles is considered.
\end{itemize}

\begin{lemma}\label{lem:tim}
	Consider \lintimveh{} for $\lambda_{\TT} =1$ and $\lambda_{\cost}=0$. Let $\femin >0$, $e \in E$, and let bounds $L_a, U_a$, $a \in \cA$, be given.
	
	Let $U_a-L_a \leq T-1$, $a \in \cA$, $L_a \geq 1$ for all $a \in \cA$ with $L_a \neq U_a$. If a sequential solution exists, the price of sequentiality (when minimizing the travel time $f_\TT$) is bounded by $\max \{T-1, \max_{w \in V} \{U_w^{\trans}, U_w^{\wait}\}\}$. 
\end{lemma}

\begin{proof}
	As the number of vehicles is unlimited and $\lambda_\cost=0$, we can construct a feasible vehicle schedule for any feasible line plan and timetable with passenger routes, e.g., by covering each trip by a single vehicle. We therefore only have to consider the line plan, passenger routes and timetable.
	
	Note that we can rewrite the objective as
	
	\begin{align*}
		& \sum_{(u,v) \in \OD} C_{u,v} \cdot \sum_{a=(i,j) \in \cA} 
		p_{a}^{u,v}\cdot (\pi_j - \pi_i + z_a \cdot T) \\
		 &= \sum_{(u,v) \in \OD} C_{u,v} \cdot \sum_{a \in \textup{SP}_{u,v}(\pi)} \textup{dur}_a(\pi)
	\end{align*}	
	
	where $\textup{SP}_{u,v}(\pi)$ is the shortest path for OD pair $(u,v)$ according to timetable $\pi$ and $\textup{dur}_a(\pi) = (\pi_j - \pi_i + z_a \cdot T)$, $a=(i,j) \in \cA$.	
        
	Recall that the bounds $L_a,U_a$, $a \in \cA$, are given according to the PTN,  
	see \eqref{explain-bounds} in the appendix and Table~\ref{tab:param:all}.
	 Therefore, the duration of a route of OD pair $(u,v)$ can be bounded from below
        by the length of a shortest $u-v$ path $R^{u,v}$ according to edge weights $L_{e}^{\drive}$ and node weights $\min\{L_w^{\wait},L_w^{\trans}\}$.
        
        This especially holds for the shortest route $\textup{SP}_{u,v}(\pi)$ for feasible timetable $\pi$. 
Conversely, an upper bound for the travel time is obtained by computing the length of $R^{u,v}$ using the upper bounds as the existence of a route corresponding to $R^{u,v}$ is given by $\femin > 0$, $e \in E$. Note that if $\femin=0$ for some PTN edge $e$, there can be optimal solutions that are not using edge $e$ resulting in a higher upper bound for the passenger route.
		
	Let $\bar{\xi}=(\bar{y},\bar{p},\bar{\pi},\bar{z},\bar{x})$ be an optimal solution of the sequential planning process and $\xi^*=(y^*,p^*,\pi^*,z^*,x^*)$ be an optimal integrated solution. Without loss of generality, we assume that $R^{u,v}$ does not contain an edge $e$ with $L_e^{\drive}=U_{e}^{\drive}=0$ or a node $w$ with $L_w^{\trans}=U_w^{\trans}=L_w^{\wait}=U_w^{\wait}=0$. Therefore, we get
	\begin{align*}
		PoS & = \frac{f(\bar{\xi})-f(\xi^*)}{f(\xi^*)} \\
		& \leq \frac{\sum\limits_{(u,v) \in \OD} C_{u,v} \cdot \sum\limits_{a \in \textup{SP}_{u,v}(\bar{\pi})} \textup{dur}_a(\bar{\pi})- \sum\limits_{(u,v) \in \OD} C_{u,v} \cdot \sum\limits_{a \in \textup{SP}_{u,v}(\pi^*)} \textup{dur}_a(\pi^*)}{\sum\limits_{(u,v) \in \OD} C_{u,v} \cdot \sum\limits_{a \in \textup{SP}_{u,v}(\pi^*)} \textup{dur}_a(\pi^*)}\\
		& \leq \frac{\sum\limits_{(u,v) \in \OD} C_{u,v} \cdot ( \sum\limits_{e \in R^{u,v}} U_e^{\drive} - L_e^{\drive}+\sum\limits_{w \in R^{u,v}} \max \{U_w^{\wait},U_w^{\trans}\} - \min \{ L_w^{\trans}, L_w^{\wait}\})}{\sum\limits_{(u,v) \in \OD} C_{u,v} \cdot ( \sum\limits_{e \in R^{u,v}} L_e^{\drive}+\sum\limits_{w \in R^{u,v}} \min \{ L_w^{\trans}, L_w^{\wait}\})}\\
		& \overset{(\star)}{\leq} \frac{\sum\limits_{(u,v) \in \OD} C_{u,v} \cdot ( (T-1) \cdot |\{e \in R^{u,v}\}| + \max\limits_{w \in V} \{U_w^{\trans},U_w^{\wait}\} \cdot |\{w \in R^{u,v}\}|)  
		}{\sum\limits_{(u,v) \in \OD} C_{u,v} \cdot ( |\{e \in R^{u,v}\}| +  |\{w \in R^{u,v}\}|)}\\
		&\leq \max \{T-1, \max_{w \in V} \{U_w^{\trans}, U_w^{\wait}\}\}
	\end{align*}
	
	where equation $(\star)$ holds as $L_a \geq 1$. 
\end{proof}

\begin{lemma}\label{lem:veh}
	Consider \lintimveh{} for $\lambda_{\TT} =0, \lambda_{\cost}=1$  where the costs $f_\cost$ only depend on the number of vehicles used. 

	If a sequential solution exists, the price of sequentiality (when minimizing the number of vehicles) is bounded by $|\cL^0|\cdot|\cP|-1$. 
\end{lemma}

\begin{proof}
	In any feasible solution to \lintimveh, the number of lines operated is bounded by $|\cL^0|$. To operate a line for all periods $\cP$ at most $|\cP|$ vehicles are needed. On the other hand, at least one vehicle is needed to operate any feasible line plan.  For an optimal sequential solution $\bar{\xi}$ and an optimal integrated solution $\xi^*$ we therefore get
	\begin{align*}
		PoS = \frac{f(\bar{\xi})-f(\xi^*)}{f(\xi^*)} \leq \frac{|\cL^0|\cdot|\cP|-1}{1}.
	\end{align*}
\end{proof}

With Theorem~\ref{thm:bound-pos-single-objective} and Lemma~\ref{lem:tim} and \ref{lem:veh}, we can bound the price of sequentiality for \lintimveh{} under mild assumptions.

\begin{corollary}
	Consider \lintimpass{} for $\lambda_\TT, \lambda_\cost \geq 0$. Let $\femin >0$, $e \in E$, and let bounds $L_a, U_a$, $a \in \cA$, be given as in Lemma~\ref{lem:tim}. Let the costs $f_\cost$ be only determined by the number of vehicles. If a sequential solution exists, the price of sequentiality is bounded by 
	\[PoS \leq \max \{T-1, \max_{w \in V} \{U_w^{\trans}, U_w^{\wait}\}, |\cL^0|\cdot|\cP|-1\} .\]
\end{corollary}

\section{Partial integration}\label{sec:partially_integrated}

Usually, the complexity of solving \msp{} is high, in particular if \msp{} is the integrated optimization
problem of a sequential process which consists of NP-complete sequential problems $(\seq_i)$.
This is in particular the case of public transport optimization where solving all sequential optimization problems in an integrated manner is not possible in reasonable time for realistically sized instances. Therefore, it might be beneficial to solve only some of the sequential problems in an integrated manner as part of the sequential solution process. This will be called
\emph{partial integration}. Partial integration is hence an intermediate step between the sequential solution approach and the integrated solution approach for \msp. We continue with the notation introduced in Section~\ref{sec-notation}.

\begin{definition} \label{def:partial}
  Let $(\seq_i\xbar)$, $i \in \{1, \ldots, n\}$, be a sequential process (as in Definition~\ref{def:sequential})
  and \msp\ be the corresponding integrated problem with $\lambda_1, \ldots, \lambda_n \geq 0$.
  Then the \emph{partially integrated problem} of stages $k$ to $l$, $1 \leq k \leq l \leq n$ is defined as
	\begin{align*}
          (\Int_{k,l}(\xbari{k-1})) \quad  \min\;
          & f_{k,l}(\xbari{k-1};x_k, \ldots, x_l)\\
          & = \sum_{i=k}^l \lambda_i \cdot f_i(\xbari{k-1},x_k,\ldots, x_{i-1}, x_i)\\
\st \quad x_k & \in F_k(\xbari{k-1})\\
	x_{k+1} & \in F_{k+1}(\xbari{k-1}, x_k)\\
			& \vdots \\
			x_l & \in F_l(\xbari{k-1},x_k, \ldots, x_{l-1}).
        \end{align*}
  \end{definition}

Analogously to Lemma~\ref{Int-better-than-Seq}    
        we receive

\begin{lemma}
\label{partial-Int-better-than-Seq}    
Let $(\bar{x}^*_k,\ldots,\bar{x}^*_l)$ be an optimal solution to 
$(\Int_{k,l}(\bar{x}_1,\ldots,\bar{x}_{k-1}))$ 
and let $\bar{x}_i$ be an optimal solution to
$(\seq_i(\bar{x}_1,\ldots,\bar{x}_{i-1}))$, $i\in\{k,\ldots,l\}$.

Then
$f_{k,l}(\bar{x}^*_k,\ldots,\bar{x}^*_l)  \leq
f_{k,l}(\bar{x}_k,\ldots,\bar{x}_l)$.
\end{lemma}

\begin{proof} The sequential solution $(\bar{x}_k,\ldots,\bar{x}_l)$ is feasible for
  $(\Int_{k,l}(\bar{x}_1,\ldots,\bar{x}_{k-1}))$,  hence the result follows.
\end{proof}
\medskip

      Let algorithms $\aA_1,\ldots,\aA_n$ for $(\seq_i)$ be known and an algorithm $\aA_{k,l}$ for $(\Int_{k,l})$,
        $1 \leq k \leq l \leq n$.
        We furthermore assume that 
\begin{itemize}
  \item for each $i\in\{1,\ldots,k-1\}$, $(\seq_i(\bar{x}_1, \ldots, \bar{x}_{i-1}))$ is feasible and $\aA_i$ finds an optimal solution $\bar{x}_i$,
  \item $(\Int_{k,l}(\bar{x}_1,\ldots,\bar{x}_{k-1}))$ is feasible and $\aA_{k,l}$ finds an optimal solution \linebreak $\PartOpt{k,l}:=(\bar{x}_k^*,\ldots,\bar{x}^*_l)$, 
  and that
  \item  for each $i\in\{l+1,\ldots,n\}$, $(\seq_i(\bar{x}_1, \ldots, \bar{x}_{k-1}, \bar{x}_k^*, \ldots, \bar{x}^*_l, \bar{x}_{l+1}, \ldots, \bar{x}_{i-1}))$ is feasible and $\aA_i$ finds an optimal solution $\bar{x}_i$.
\end{itemize}    
Then we call
\begin{eqnarray*}
  \Part{k,l} &:= &(\bar{x}_1,\ldots,\bar{x}_{k-1},\underbrace{\bar{x}_k^*,\ldots, \bar{x}^*_l,}_{{\rm optimal\ for}\ (\Int_{k,l})}\bar{x}_{l+1},\ldots, \bar{x}_n) \\
  & = & (\bar{x}_1,\ldots,\bar{x}_{k-1},\PartOpt{k,l},\bar{x}_{l+1},\ldots, \bar{x}_n)
\end{eqnarray*}        
a \emph{partially integrated solution} w.r.t.\ $(\Int_{k,l})$.
Note that $\Part{k,k}=\bar{x}$ and $\Part{1,n}=x^*$.

Concerning the notation, we remark that subindices refer to components of a solution,
  i.e., $x_i \in \R^{m_i}$ are the variables of the $i$th stage and $x_{k,l} \in \R^{m_k} \times \ldots \times \R^{m_l}$
  are the variables of stages $k$ to $l$. Conversely, superindices refer to complete solutions, i.e., $\bar{x}^{k,l}$ is
  a complete solution vector which has been built by partial integration as described in the algorithm below.

The corresponding solution approach is to solve stages $1$ to $k-1$ sequentially by algorithms ${\aA}_1, \ldots, {\aA}_{k-1}$, then use algorithm ${\aA}_{k,l}$ to solve
$(\Int_{k,l}(\xbari{k-1}))$, and to continue with solving stages $l+1$ to $n$ by algorithms ${\aA}_{l+1},\ldots,{\aA}_n$.
The partially integrated approach is shown in Figure~\ref{fig:partially_integrated}.
\smallskip

\begin{center}
\small{
\begin{tabular}{llll}
		\toprule
		\textbf{Problem} & \textbf{Algorithm} & \textbf{Input} & \textbf{Output}\\
		\midrule
 		$(\seq_1)$ & ${\aA}_1$ &  $\emptyset$ & $\bar{x}_1$\\
             & \vdots & \\
        $(\seq_{k-1})$ & ${\aA}_{k-1}$ &  $\bar{x}_1,\ldots,\bar{x}_{k-2}$ & $\bar{x}_{k-1}$\\
        $(\Int_{k,l})$ & ${\aA}_{k,l}$ &  $\bar{x}_1,\ldots,\bar{x}_{k-1}$ & $\bar{x}^*_{k,l}=(\bar{x}^*_{k},\ldots,\bar{x}^*_l)$\\
  $(\seq_{l+1})$ & ${\aA}_{l+1}$ &
$\bar{x}_1,\ldots,\bar{x}_{k-1},\bar{x}^*_{k,l}$ & $\bar{x}_{l+1}$\\
        & \vdots &  \\
  $(\seq_n)$ & ${\aA}_n$ &
  $\bar{x}_1,\ldots,\bar{x}_{k-1},\bar{x}^*_{k,l}, \bar{x}_{l+1},\ldots \bar{x}_{n-1}$ & $\bar{x}_n$\\
        \bottomrule
      \end{tabular}}
    \end{center}

	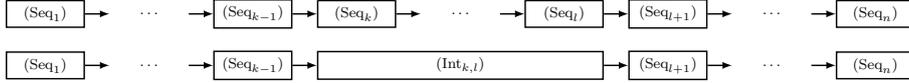
\begin{figure}[h]
	\begin{center}
	\resizebox{\textwidth}{!}{
		\begin{tikzpicture}[thick, scale=1, font= \footnotesize]
	\node[draw, rectangle, minimum width=1.5cm] (seq_1_1) at (0,0) {$(\seq_1)$};
	\node[minimum width=1.5cm] (seq_2_1) at (2,0) {$\ldots$};
	\node[draw, rectangle, minimum width=1.5cm] (seq_k1_1) at (4,0) {$(\seq_{k-1})$};
	\node[draw, rectangle, minimum width=1.5cm] (seq_k_1) at (6,0) {$(\seq_{k})$};
	\node[minimum width=1.5cm] (seq_k2_1) at (8,0) {$\ldots$};
	\node[draw, rectangle, minimum width=1.5cm] (seq_l_1) at (10,0) {$(\seq_l)$};
	\node[draw, rectangle,  minimum width=1.5cm] (seq_l1_1) at (12,0) {$(\seq_{l+1})$};
	\node[minimum width=1.5cm] (seq_l2_1) at (14,0) {$\ldots$};
	\node[draw, rectangle, minimum width=1.5cm] (seq_n_1) at (16,0) {$(\seq_{n})$};
	\node[draw, rectangle, minimum width=1.5cm] (seq_1_2) at (0,-1) {$(\seq_1)$};
	\node[minimum width=1.5cm] (seq_2_2) at (2,-1) {$\ldots$};
	\node[draw, rectangle, minimum width=1.5cm] (seq_k1_2) at (4,-1) {$(\seq_{k-1})$};
	\node[draw, rectangle, minimum width=5.5cm] (int_kl) at (8,-1) {$(\Int_{k,l})$};
	\node[draw, rectangle, minimum width=1.5cm] (seq_l1_2) at (12,-1) {$(\seq_{l+1})$};
	\node[minimum width=1.5cm] (seq_l2_2) at (14,-1) {$\ldots$};
	\node[draw, rectangle, minimum width=1.5cm] (seq_n_2) at (16,-1) {$(\seq_{n})$};
	\draw[-latex] (seq_1_1) to (seq_2_1);
	\draw[-latex] (seq_2_1) to (seq_k1_1);
	\draw[-latex] (seq_k1_1) to (seq_k_1);
	\draw[-latex] (seq_k_1) to (seq_k2_1);
	\draw[-latex] (seq_k2_1) to (seq_l_1);
	\draw[-latex] (seq_l_1) to (seq_l1_1);
	\draw[-latex] (seq_l1_1) to (seq_l2_1);
	\draw[-latex] (seq_l2_1) to (seq_n_1);
	\draw[-latex] (seq_1_2) to (seq_2_2);
	\draw[-latex] (seq_2_2) to (seq_k1_2);
	\draw[-latex] (seq_k1_2) to (int_kl);
	\draw[-latex] (int_kl) to (seq_l1_2);
	\draw[-latex] (seq_l1_2) to (seq_l2_2);
	\draw[-latex] (seq_l2_2) to (seq_n_2);
\end{tikzpicture}
		}
		\caption{Partial integration of stages $k$ to $l$ in a sequential solution approach.}\label{fig:partially_integrated}
	\end{center}
	\end{figure}

        We illustrate the meaning of partial integration using again public transport optimization. Various partially integrated solution approaches are presented in Figure~\ref{fig:ptp}. The sequential solution approach, depicted in the first row, refers to solving line planning, passenger routing, timetabling and vehicle scheduling separately as described in Section~\ref{sec:seq_sol}. The integrated solution approach of solving \lintimveh{} as described in Section~\ref{sec:model}, is depicted in the last row. The second row shows the partially integrated solution process
        w.r.t.\ \timpass{} or $(\Int_{2,3})$.
        The partially integrated solution approach w.r.t.\ line planning, timetabling and passenger routing as \lintimpass{} or $(\Int_{1,3})$
        is depicted in row~3.  Row~4 shows the partially integrated solution approach
w.r.t.\ \timveh{} or $(\Int_{3,4})$.
For detailed descriptions of these models, see \cite{philinediss}.

\begin{figure}
\begin{center}
\begin{center}
\resizebox{0.85\textwidth}{!}{
	\begin{tikzpicture}[thick, font=\scriptsize]
		\node[draw, align=center, color=gray, text width=1.5cm, minimum height=2.5em] (ND) at (-10,0) {network design};
		\node[draw, align=center, text width=1.5cm, minimum height=2.5em] (LP) at (-8,0) {line planning};
		\node[draw, align=center, text width=1.5cm, minimum height=2.5em] (PassEAN) at (-6,0) {passenger routing};
		\node[draw, align=center, text width=1.5cm, minimum height=2.5em] (TT) at (-4,0) {timetabling};
		\node[draw, align=center, text width=1.5cm, minimum height=2.5em] (VS) at (-2,0) {vehicle scheduling};
		\node[draw, align=center, color=gray, text width=1.5cm, minimum height=2.5em] (CS) at (0,0) {crew scheduling};
		\draw[-latex, color=gray] (ND) edge (LP);
		\draw[-latex] (LP) edge  (PassEAN);
		\draw[-latex] (PassEAN) edge (TT);
		\draw[-latex] (TT) edge (VS);
		\draw[-latex, color=gray] (VS) edge (CS);
		\node[draw, align=center, color=gray, text width=1.5cm, minimum height=2.5em] (1ND) at (-10,-1) {network design};
		\node[draw, align=center, text width=1.5cm, minimum height=2.5em,color=amber] (1LP) at (-8,-1) {line planning};
		\node[draw, align=center, text width=3.5cm, minimum height=2.5em,color=amber] (1TimPass) at (-5,-1) {\timpass};
		\node[draw, align=center, text width=1.5cm, minimum height=2.5em,color=amber] (1VS) at (-2,-1) {vehicle scheduling};
		\node[draw, align=center, color=gray, text width=1.5cm, minimum height=2.5em] (1CS) at (0,-1) {crew scheduling};
		\draw[-latex, color=gray] (1ND) edge (1LP);
		\draw[-latex,color=amber] (1LP) edge  (1TimPass);
		\draw[-latex,color=amber] (1TimPass) edge (1VS);
		\draw[-latex, color=gray] (1VS) edge (1CS);
		\node[draw, align=center, color=gray, text width=1.5cm, minimum height=2.5em] (2ND) at (-10,-2) {network design};
		\node[draw, align=center, text width=5.5cm, minimum height=2.5em, color=red] (2LinTimPass) at (-6,-2) {\lintimpass};
		\node[draw, align=center, text width=1.5cm, minimum height=2.5em, color=red] (2VS) at (-2,-2) {vehicle scheduling};
		\node[draw, align=center, color=gray, text width=1.5cm, minimum height=2.5em] (2CS) at (0,-2) {crew scheduling};
		\draw[-latex, color=gray] (2ND) edge (2LinTimPass);
		\draw[-latex, color=red] (2LinTimPass) edge (2VS);
		\draw[-latex, color=gray] (2VS) edge (2CS);
		\node[draw, align=center, color=gray, text width=1.5cm, minimum height=2.5em] (3ND) at (-10,-3) {network design};
		\node[draw, align=center, text width=1.5cm, minimum height=2.5em,color=deepmagenta] (3LP) at (-8,-3) {line planning};
		\node[draw, align=center, text width=1.5cm, minimum height=2.5em,color=deepmagenta] (3PassEAN) at (-6,-3) {passenger routing};
		\node[draw, align=center, text width=3.5cm, minimum height=2.5em,color=deepmagenta] (3TimVeh) at (-3,-3) {\timveh};
		\node[draw, align=center, color=gray, text width=1.5cm, minimum height=2.5em] (3CS) at (0,-3) {crew scheduling};
		\draw[-latex, color=gray] (3ND) edge (3LP);
		\draw[-latex,color=deepmagenta] (3LP) edge  (3PassEAN);
		\draw[-latex,color=deepmagenta] (3PassEAN) edge (3TimVeh);
		\draw[-latex, color=gray] (3TimVeh) edge (3CS);
		\node[draw, align=center, color=gray, text width=1.5cm, minimum height=2.5em] (4ND) at (-10,-4) {network design};
		\node[draw, align=center, text width=7.5cm, minimum height=2.5em, color=darkgreen] (4LTPV) at (-5,-4) {\lintimpass};
		\node[draw, align=center, color=gray, text width=1.5cm, minimum height=2.5em] (4CS) at (0,-4) {crew scheduling};
		\draw[-latex, color=gray] (4ND) edge (4LTPV);
		\draw[-latex, color=gray] (4LTPV) edge (4CS);
	\end{tikzpicture}}
      \end{center}
\caption{The different partially integrated solution approaches.}\label{fig:ptp}
\end{center}
\end{figure}
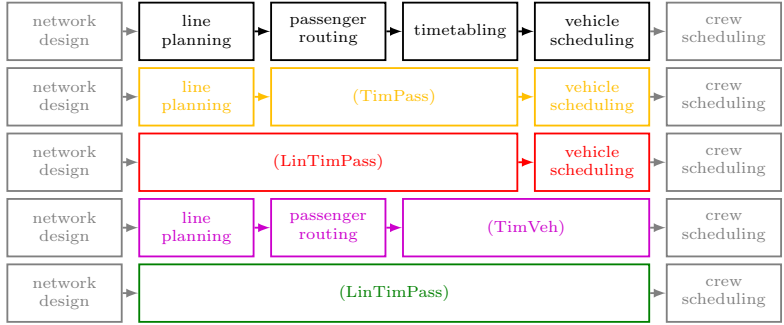

Before we further analyze partial integration we define a price of sequentiality for partially integrated solutions.

\begin{definition}
  Let $x^*$ be optimal for \msp{} with $f(x^*) >0$ and let
$\Part{k,l}$ be optimal w.r.t.\ $(\Int_{k,l})$.
Then the \emph{price of sequentiality for the partially integrated approach} w.r.t.\ $(\Int_{k,l})$ is defined as
	\[\PoS_{k,l} = \frac{f(\Part{k,l}) - f(x^*)}{f(x^*)}.\] 
\end{definition}

Note that Lemma~\ref{partial-Int-better-than-Seq} says that solving $(\Int_{k,l})$ leads to better solutions than solving the stages $k,k+1,\ldots,l$ sequentially.
As we will see in the following example, this does not guarantee that the resulting partially integrated solution is better than
the sequential one, i.e., the price of sequentiality might increase for some partially integrated approaches compared to the
sequential approach of solving all stages individually. This can also be seen in the experiments in Section~\ref{sec:experiments}
and in further examples in \cite{philinediss}.

\begin{example}
We extend Example~\ref{ex:pos_unbounded} by adding a third stage. Consider the sequential problems
\begin{align*}
	(\seq_1) \quad &  \{\min x_1 \colon x_1 \leq 1, x_1 \geq 0\}\\
	(\seq_2(\bar{x}_1)) \quad & \{\min  x_2 \colon x_2 \leq 1 -\bar{x}_1, x_2 \geq 1 - N \cdot \bar{x}_1,  x_2 \geq 0\},\\
	(\seq_3(\bar{x}_1,\bar{x}_2)) \quad & \{\min  x_3 \colon x_3 \geq N^2 \cdot \bar{x}_1,  x_3 \geq 0\}.
\end{align*}
For $N\geq 0$, the optimal sequential solution is $\bar{x}=(0,1,0)$ with objective value $1$. 

For $\lambda=(1,1,1)^t$ we know from Example~\ref{ex:pos_unbounded}, that the optimal solution of $(\Int_{1,2})$ is $\PartOpt{1,2}=(\frac{1}{N},0)$. Therefore, the partially integrated solution w.r.t.\ $(\Int_{1,2})$ is $\Part{1,2}=(\frac{1}{N},0,N)$ with objective value $\frac{N^2+1}{N}$.  

For the corresponding integrated problem \msp, i.e., for

\begin{mini*}
	{}{x_1 + x_2 + x_3}{}{\msp \;}
	\addConstraint{x_1 +x_2 }{\leq 1}{}
	\addConstraint{N\cdot x_1 + x_2 }{\geq 1}{}
	\addConstraint{-N^2 \cdot x_1 +x_3 }{\leq 0}{}
	\addConstraint{x_1, x_2,x_3 }{\geq 0}{}
\end{mini*}

the optimal  solution is $x^*=(0,1,0)$, the same as the sequential solution $\bar{x}$. Thus, the price of sequentiality is zero while  the price of sequentiality of the partially integrated approach w.r.t.\ $(\Int_{1,2})$ satisfies
\[\PoS = \frac{\frac{N^2+1}{N} -1}{1} = N + \frac{1}{N} \overset{N \to \infty}{\to} \infty.\]
\end{example}

In the next theorem we show in which case we can guarantee better solutions through a partially integrated approach, namely, when the problems to integrate are chosen properly.

    \begin{theorem}\label{thm:partial_objective}
      Let $(\seq_i(\xbari{i-1})), i\in \{1,\ldots,n\}$, be a family of sequential problems with sequential solution $\bar{x}$, obtained
      by algorithms $\aA_1,\ldots,\aA_n$ in the sequential solution approach. Let $x^*$ be an optimal solution for the
      corresponding multi-stage problem \msp.  Consider partial integration w.r.t.\ $(\Int_{k,n}(\xbari{k-1}))$, i.e., for the last $n-k+1$
      stages, and let $\Part{k,n}$ be the corresponding partially integrated solution. Then we have
      \[ f(x^*) = f(\Part{1,n}) \leq f(\Part{2,n}) \leq \ldots \leq f(\Part{n-1,n}) \leq f(\Part{n,n})=f(\bar{x}).\]
    \end{theorem}

    \begin{proof}
      Per definition we have $x^*=\Part{1,n}$ and $\bar{x}=\Part{n,n}$. It remains to show that
      $f(\Part{k,n}) \leq f(\Part{k+1,n})$ for all $k\in \{1,\ldots,n-1\}$. To this end, we compare
      \begin{eqnarray*}
        \Part{k,n}   & = & (\bar{x}_1,\ldots,\bar{x}_{k-1},\PartOpt{k,n}) \mbox{ and }\\
        \Part{k+1,n} & = & (\bar{x}_1,\ldots,\bar{x}_{k-1},\bar{x}_k,\PartOpt{k+1,n}),
      \end{eqnarray*}
      see Figure~\ref{fig:multistage:intk} for an illustration.
      Since $(\bar{x}_k,\PartOpt{k+1,n})$ is feasible for $(\Int_{k,n}(\bar{x}_1,\ldots,\bar{x}_{k-1}))$
      we obtain 
      \[ f_{k,n}(\PartOpt{k,n}) \leq f_{k,n}(\bar{x}_k,\PartOpt{k+1,n}),\]
      and hence
      \begin{eqnarray*}
      f(\Part{k,n}) & = & \sum_{i=1}^{k-1} \lambda_i \cdot f_i(\bar{x}_1,\ldots,\bar{x}_i) + f_{k,n}(\PartOpt{k,n}) \\
             & \leq & \sum_{i=1}^{k-1} \lambda_i \cdot f_i(\bar{x}_1,\ldots,\bar{x}_i) + f_{k,n}(\bar{x}_k,\PartOpt{k+1,n})\\
             &= & \sum_{i=1}^{k-1} \lambda_i \cdot f_i(\bar{x}_1,\ldots, \bar{x}_i) + \lambda_k \cdot  f_k(\bar{x}_1,\ldots, \bar{x}_k) + f_{k+1,n}(\PartOpt{k+1,n}) \\ 
              & = & f(\Part{k+1,n}) 
      \end{eqnarray*}
    \end{proof}
 
From Theorem~\ref{thm:partial_objective}, we conclude that we can decrease the price of sequentiality by
    integrating the \emph{last} stages in the solution process. 

\begin{corollary}\label{cor:pos_partially _integrated}
	\[\PoS_{n-1,n} \geq  \PoS_{n-2,n} \geq \ldots \geq \PoS_{2,n} \geq \PoS_{1,n} = 0.\] 
\end{corollary}

We observe this effect also in our experiments on partially integrated problems in public transport planning in Section~\ref{sec:experiments}. When integrating the last two planning stages timetabling and vehicle scheduling as \timveh{}, the resulting solution always is at least as good as the sequential solution approach.

In Theorem~\ref{thm:partial_objective} we started with a sequential solution $\bar{x}$ in which each 
  stage $(\seq_i)$ has been solved optimally. Although treating the problems sequentially is computationally much easier
  than solving \msp, even solving the sequential problems might be too hard in applications. This is,  e.g., the case
  for the timetabling stage in public transport optimization. In the following we hence present a generalization of
  Theorem~\ref{thm:partial_objective} in which we do not require optimality of the sequential solution. To this
  end, we allow the algorithms $\aA_1,\ldots,\aA_n$ to be heuristics instead of exact optimization algorithms.
  We obtain the following result.

\begin{theorem}
  Let $(\seq_i(\tilde{x}_1, \ldots, \tilde{x}_{i-1}))$, $i\in \{1,\ldots,n\}$, be a family of sequential problems with a (possibly) heuristic sequential
  solution $\tilde{x}$, obtained by heuristics or exact algorithms $\tilde{\aA}_1,\ldots,\tilde{\aA}_n$ in the sequential solution approach.
  As before, let $x^*$ be an optimal solution for the corresponding multi-stage problem \msp.

  Consider the following partial integration of the \emph{heuristic} sequential process  w.r.t.\ $(\Int_{k,n}(\tilde{x}_1, \ldots, \tilde{x}_{k-1}))$:
  We apply the heuristics $\tilde{\aA}_1,\ldots,\tilde{\aA}_{k-1}$ with solutions $\tilde{x}_1,\ldots,\tilde{x}_{k-1}$ followed
  by the \emph{exact} algorithm $\aA_{k,n}$
  solving $(\Int_{k,n}(\tilde{x}_1,\ldots,\tilde{x}_{k-1}))$. 
  Let $\tilde{x}^{k,n}$ be the corresponding heuristic partially integrated solution. Then we have
  \[ f(x^*) = f(\tilde{x}^{1,n}) \leq f(\tilde{x}^{2,n}) \leq \ldots \leq f(\tilde{x}^{n-1,n}) \leq f(\tilde{x}^{n,n})
      =f(\tilde{x}).\]
\end{theorem}

\begin{proof}
Since optimality   of the sequential solution has not been used in the proof of Theorem~\ref{thm:partial_objective}, we can re-use it by replacing $\bar{x}$ by $\tilde{x}$.
\end{proof}

When solving problems in public transport planning such as line planning, passenger routing, timetabling and
vehicle scheduling, we hence can still use the result of
Theorem~\ref{thm:partial_objective}
to improve the planning process in public transport planning
even if the planning stages themselves cannot be solved to optimality.

 \begin{figure}[h]
	\begin{center}
	\resizebox{\textwidth}{!}{
		\begin{tikzpicture}[thick, scale=1, font= \footnotesize]
	\node[draw, rectangle, minimum width=1.5cm] (seq_1_1) at (0,0) {$(\seq_1)$};
	\node[draw, rectangle,minimum width=1.5cm] (seq_2_1) at (2,0) {$(\seq_2)$};
	\node[minimum width=1.5cm] (seq_3_1) at (4,0) {$\ldots$};
	\node[draw, rectangle, , minimum width=1.5cm] (seq_k1_1) at (6,0) {$(\seq_{k-1})$};
	\node[draw, rectangle, minimum width=1.5cm] (seq_k_1) at (8,0) {$(\seq_k)$};
	\node[draw, rectangle, minimum width=3.5cm] (int_1) at (11,0) {$(\Int_{k+1,n})$};
		\node[draw, rectangle, minimum width=1.5cm] (seq_1_2) at (0,-1) {$(\seq_1)$};
	\node[draw, rectangle, minimum width=1.5cm] (seq_2_2) at (2,-1) {$(\seq_2)$};
	\node[minimum width=1.5cm] (seq_3_2) at (4,-1) {$\ldots$};
	\node[draw, rectangle,  minimum width=1.5cm] (seq_k1_2) at (6,-1) {$(\seq_{k-1})$};
	\node[draw, rectangle, minimum width=5.5cm] (int_2) at (10,-1) {$(\Int_{k,n})$};
	\draw[-latex] (seq_1_1) to (seq_2_1);
	\draw[-latex] (seq_2_1) to (seq_3_1);
	\draw[-latex] (seq_3_1) to (seq_k1_1);
	\draw[-latex] (seq_k1_1) to (seq_k_1);
	\draw[-latex] (seq_k_1) to (int_1);
	\draw[-latex] (seq_1_2) to (seq_2_2);
	\draw[-latex] (seq_2_2) to (seq_3_2);
	\draw[-latex] (seq_3_2) to (seq_k1_2);
	\draw[-latex] (seq_k1_2) to (int_2);
\end{tikzpicture}
		}
		\caption{Integrating the last $n-k$ or $n-k+1$ stages of \msp.}\label{fig:multistage:intk}
	\end{center}
	\end{figure}
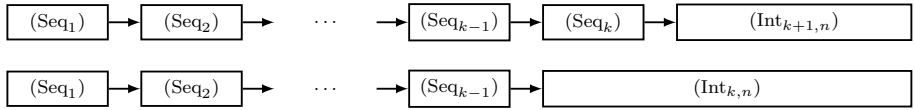	

\section{Experiments}\label{sec:experiments}
\label{sec-experiments}

In this section we experimentally evaluate the price of sequentiality for (partially) integrating the problems line planning, passenger routing, timetabling and vehicle scheduling in public transport planning. We use two small artificial data sets \texttt{small} and \texttt{toy} from \cite{lintimhp,lintim}, see Figure~\ref{fig:datasets}, as a proof of concept on a computer with a Ryzen 5 PRO 2500U CPU \@2GHz and 16 GB RAM running Gurobi~8, \cite{gurobi8}. 

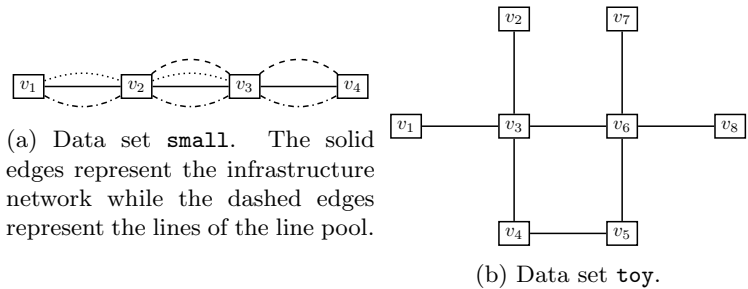
\begin{figure}[h!]
\begin{center}
\begin{subfigure}{0.4\textwidth}
\resizebox{\textwidth}{!}{
	\begin{tikzpicture}[thick, scale = 0.5]
		\node[draw, rectangle] (v1) at (0,0) {$v_1$};
		\node[draw, rectangle] (v2) at (4,0) {$v_2$};
		\node[draw, rectangle] (v3) at (8,0) {$v_3$};
		\node[draw, rectangle] (v4) at (12,0) {$v_4$};
		\draw (v1) to (v2);
		\draw (v2) to (v3);
		\draw (v3) to (v4);
		\draw[dash dot, bend right] (v1) to (v2);
		\draw[dash dot, bend right] (v2) to (v3);
		\draw[dash dot, bend right] (v3) to (v4);
		\draw[dotted, bend left=20] (v1) to (v2);
		\draw[dotted, bend left=20] (v2) to (v3);
		\draw[dashed, bend left=40] (v2) to (v3);
		\draw[dashed, bend left=40] (v3) to (v4);
	\end{tikzpicture}	
}
\caption{Data set \texttt{small}. The solid edges represent the infrastructure network while the dashed edges represent the lines of the line pool.} \label{fig:dataset:small}
\end{subfigure}
\begin{subfigure}{0.4\textwidth}
\resizebox{\textwidth}{!}{
	\begin{tikzpicture}[thick, scale = 2]
		\node[draw, rectangle] (v1) at (1,2) {$v_1$};
		\node[draw, rectangle] (v2) at (2,3) {$v_2$};
		\node[draw, rectangle] (v3) at (2,2) {$v_3$};
		\node[draw, rectangle] (v4) at (2,1) {$v_4$};
		\node[draw, rectangle] (v5) at (3,1) {$v_5$};
		\node[draw, rectangle] (v6) at (3,2) {$v_6$};
		\node[draw, rectangle] (v7) at (3,3) {$v_7$};
		\node[draw, rectangle] (v8) at (4,2) {$v_8$};
		\draw (v1) to (v3);
		\draw (v2) to (v3);
		\draw (v3) to (v4);
		\draw (v4) to (v5);
		\draw (v5) to (v6);
		\draw (v3) to (v6);
		\draw (v6) to (v7);
		\draw (v6) to (v8);
	\end{tikzpicture}	
}
\caption{Data set \texttt{toy}.} \label{fig:dataset:toy}
\end{subfigure}
\caption{Public transportation networks.}\label{fig:datasets}
\end{center}
\end{figure}

At first, we consider the matrix structure of the integrated problem \lintimveh. The total number of variables and constraints for data sets \texttt{small} and \texttt{toy} is given in Table~\ref{tab:size} while the matrix structure is presented graphically in Figure~\ref{fig:matrix_structure}. Note that the number of variables and constraints for the line planning subproblems are given explicitly as these blocks are not visible due to their size. 

The size of the blocks for the subproblems varies considerably. This is one of the reasons why decomposition the matrix according to the subproblems is not well suited for a Dantzig-Wolfe decomposition approach, see e.g.~\cite{LPSS-CASPT18} and \cite{lubbeckeprimal} for more general contexts. Especially passenger routing contributes largely to the overall size of the matrix which motivate the reduction of the solution space for passenger routing as e.g.~presented in \cite{philinediss,SchiSch18,ATMOS2021}.

\begin{table}[h!]
\small{
\begin{center}
	\begin{tabular}{lrr}
	\toprule
  & \texttt{small} & \texttt{toy} \\ 
 \midrule
 $|V|$ & 4 & 8 \\ 
 $|E|$ & 3 & 8 \\ 
 \midrule
 variables \lintimveh & 2322& 70152 \\ 
 constraints \lintimveh & 5033 & 118060 \\ 
 \bottomrule
 \end{tabular} 
\end{center}
\caption{Problem size for data sets \texttt{small} and \texttt{toy}.}\label{tab:size}}
\end{table}

\begin{figure}[h!]
\begin{center}
\begin{subfigure}{0.49\textwidth}
\resizebox{\textwidth}{!}{
\includegraphics{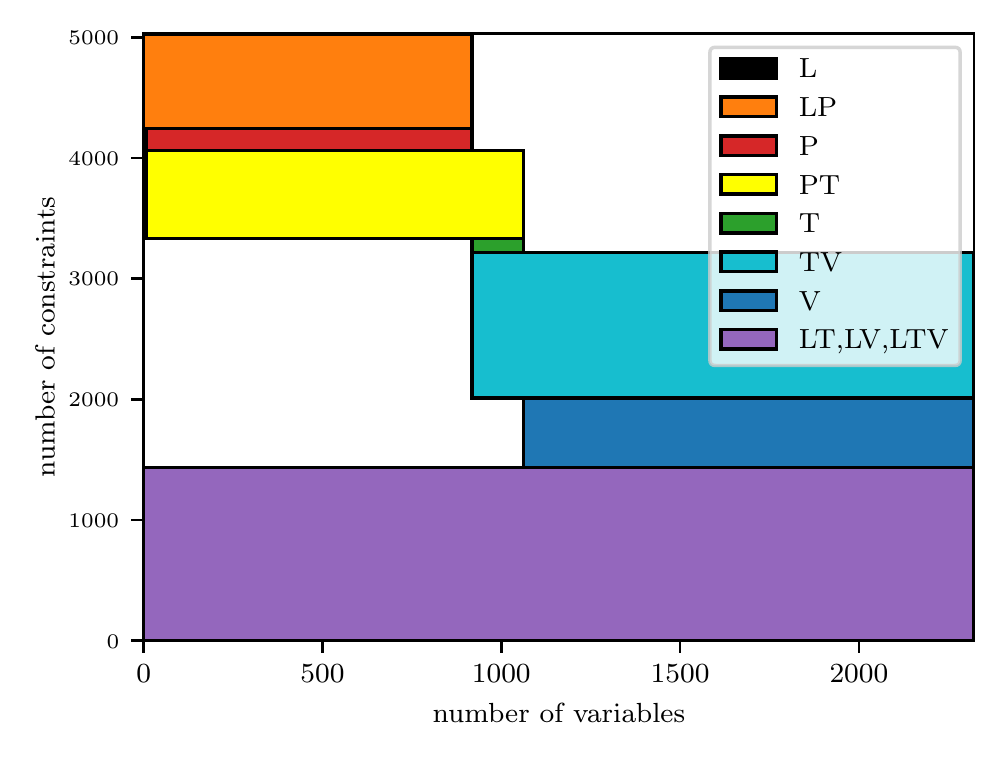}
}
\caption{Data set \texttt{small}. Total: 2322 variables, 5033 constraints. Line planning: 6 variables, 9 constraints.}
\end{subfigure}
\begin{subfigure}{0.49\textwidth}
\resizebox{\textwidth}{!}{
\includegraphics{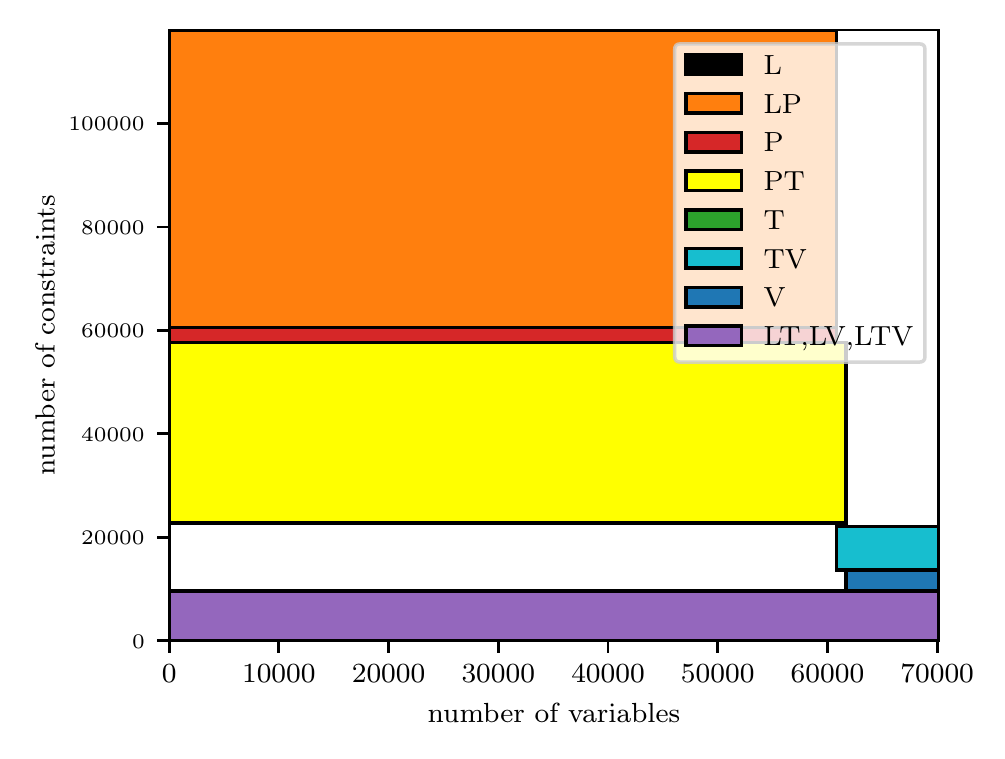}
}
\caption{Data set \texttt{toy}. Total: 70152 variables, 118060 constraints. Line planning: 16 variables, 40 constraints.}
\end{subfigure}
\end{center}
\caption{Matrix structure of the integrated line planning, passenger routing, timetabling and vehicle scheduling problem.}\label{fig:matrix_structure}
\end{figure}

For evaluating the integrated solution approach, we consider different parameters $\lambda_{\TT}, \lambda_{\cost}$ as weights for the travel time $f_3$ and the costs $f_4$ in the objective function
\[\lambda_\TT \cdot f_3 + \lambda_\cost \cdot f_4\]
given in Section~\ref{sec:model}. However, for the partially integrated model \lintimpass{}, it is often beneficial to include the line costs $f_1$ as an approximation of the costs $f_4$ and use the auxiliary objective
\[\lambda_{\linecost} \cdot f_1 + \lambda_\TT \cdot f_3.\] An overview of the parameters used in this section can be found in Table~\ref{tab:parameters}.

\begin{table}[h]
\small{
\begin{center}
\begin{tabular}{lcc}
\toprule
 \textbf{Data set} & $(\lambda_\linecost,\lambda_\TT,\lambda_\cost)$&  \textbf{Figure}\\ 
\midrule
\texttt{small} & (0, 1000, 1) & \ref{fig:small:ptp} \\
\midrule
\texttt{toy} & \makecell{(0,10,1)\\(0,1,0)\\(0,0,1)\\(0,40,1)\\(25000,40,-)$^*$\\(25000,10,-)$^*$\\(1000,40,-)$^*$\\(1000,10,-)$^*$\\(500,40,-)$^*$\\(500,10,-)$^*$} & 
\makecell{\ref{fig:toy1:weighted_sum}, \ref{fig:toy:pareto_restricted}\\
\ref{fig:toy2:weighted_sum}\\
\ref{fig:toy3:weighted_sum}, 
\ref{fig:toy:pareto_restricted}\\
\ref{fig:toy4:weighted_sum}, 
\ref{fig:toy:pareto_restricted}\\
\ref{fig:toy:pareto_restricted}\\
\ref{fig:toy:pareto_restricted}\\
\ref{fig:toy:pareto_restricted}\\
\ref{fig:toy:pareto_restricted}\\
\ref{fig:toy:pareto_restricted}\\
\ref{fig:toy:pareto_restricted}}\\
\bottomrule
\end{tabular} 
\end{center}}
\caption{Parameters used in the objective function. Parameter sets $^*$ are only used for \lintimpass.}\label{tab:parameters}
\end{table}

\paragraph*{Problem size}
As Table~\ref{tab:solvertime} shows, the time limit of one hour does not suffice to solve \lintimveh{} to optimality for all considered parameters $\lambda_{\TT}, \lambda_{\cost}$ for the small artificial data set \texttt{toy}. We therefore compare different partially integrated solution approaches as presented in Figure~\ref{fig:ptp}. The solver time of these partially integrated problems is considerably smaller than for \lintimveh{}, see Table~\ref{tab:solvertime}. This is especially obvious for data set \texttt{toy}, where the solver time of the partially integrated problems ranges between 0.003\% and 12.2\% of the solver time of \lintimveh. Additionally, Table~\ref{tab:solvertime} illustrates the amplification of the increase of the problem size compared to the data set size. While the infrastructure network for data set \texttt{toy} is roughly twice as large as the infrastructure network for data set \texttt{small}, the solver time for \lintimveh{} increases from finding an optimal solution in less than two seconds to an average gap of 50\% for the one hour time limit.

\begin{table}[h]
\begin{center}
\small{
\begin{tabular}{lrrrr}
\toprule
\multirow{2}{*}{\textbf{Solution approach}} & \multicolumn{2}{c}{\textbf{\texttt{small}}}
& \multicolumn{2}{c}{\textbf{\texttt{toy}}}\\

 & \textbf{time}& \textbf{gap} & \textbf{time} & \textbf{gap} \\ 
\midrule 
\timpass & 0.02 & 0 & 0.08 & 0 \\ 
\lintimpass & 0.35 & 0 & 331.21 & 0 \\ 
\timveh & 0.52 & 0 & 3.42 & 0 \\ 
\lintimveh & 1.35 & 0 & 2709.70 & 50.35\% \\ 
\bottomrule
\end{tabular} }
\caption{Solver time (in seconds) and optimality gap of the (partially) integrated models for data sets \texttt{small} and \texttt{toy} for a time limit of one hour. The values reported here are mean values over all considered parameters, see Table~\ref{tab:parameters}.}
\label{tab:solvertime}
\end{center}
\end{table}

\paragraph{Comparing the solution quality of partially integrated solution approaches} In Theorem~\ref{thm:partial_objective}, we show that the integration of the last stages always leads to solution that are better or at least as good as the sequential solution. 
This is also reflected by the experiments, see Figure~\ref{fig:small:pos} for data set \texttt{small} and Figure~\ref{fig:toy:ptp} for data set \texttt{toy}.
For the latter, we see that emphasizing costs in the objective, which is for $\lambda_\TT =10, \lambda_\cost=1$ or even $\lambda_\TT=0, \lambda_\cost=1$, \timveh{} leads to a considerable improvement over the sequential solution approach. We marked the solution approaches for which an improvement can be guaranteed with a gray circle.

For the other partially integrated solution approaches, there is no guarantee to find better solutions compared to the sequential approach. 
Our experiments show that \timpass{} is able to improve the solution quality, see Figures~\ref{fig:toy2:weighted_sum} to \ref{fig:toy4:weighted_sum}.
 However, \lintimpass{} with objective parameters $\lambda_\linecost =0, \lambda_\TT > 0$ leads to an impaired solution quality for all cases where for the costs $f_\cost$ the parameter is positive, i.e., $\lambda_\cost >0$. 
 Only for the case that only the travel time $f_\TT$ is considered, i.e., for $\lambda_\TT=1, \lambda_\cost=0$, see Figure~\ref{fig:toy2:weighted_sum}, the objective is as good as for \lintimveh. 
Note that in this case \lintimveh{} and \lintimpass{} coincide as a feasible vehicle schedule can always be constructed.

\begin{figure}[h!]
\begin{center}
\begin{subfigure}{0.45\textwidth}
	\begin{center}
	\resizebox{\textwidth}{!}{
		\includegraphics{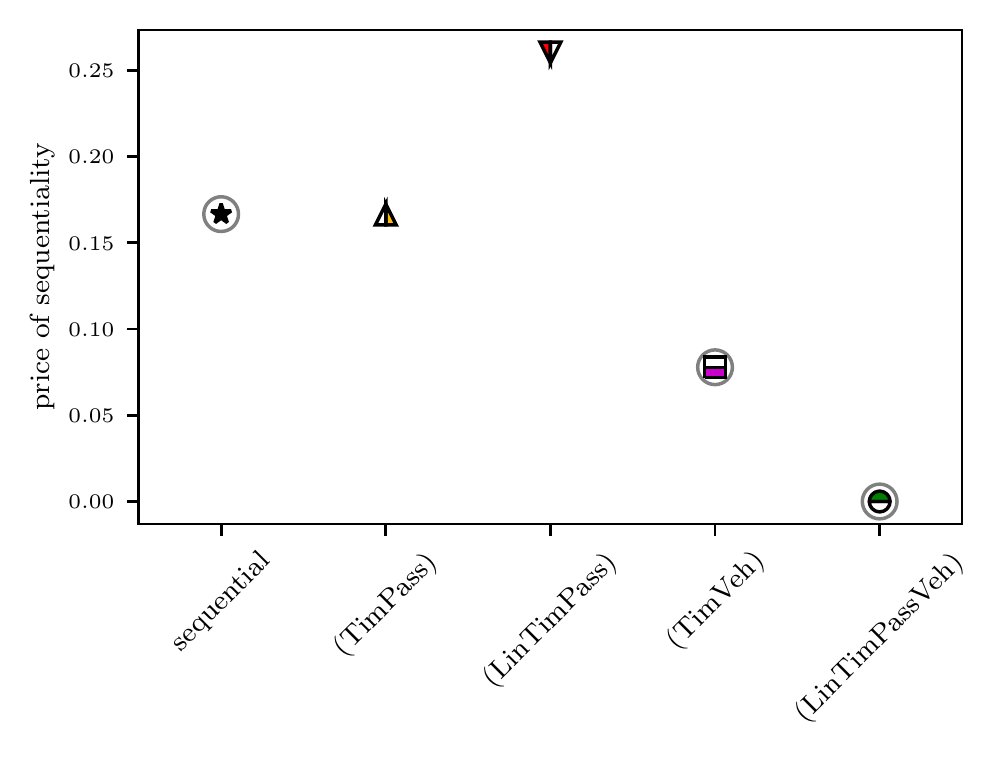}
	}
	\end{center}
	\caption{The price of sequentiality.}\label{fig:small:pos}
\end{subfigure}
\begin{subfigure}{0.45\textwidth}
	\begin{center}
	\resizebox{\textwidth}{!}{
		\includegraphics{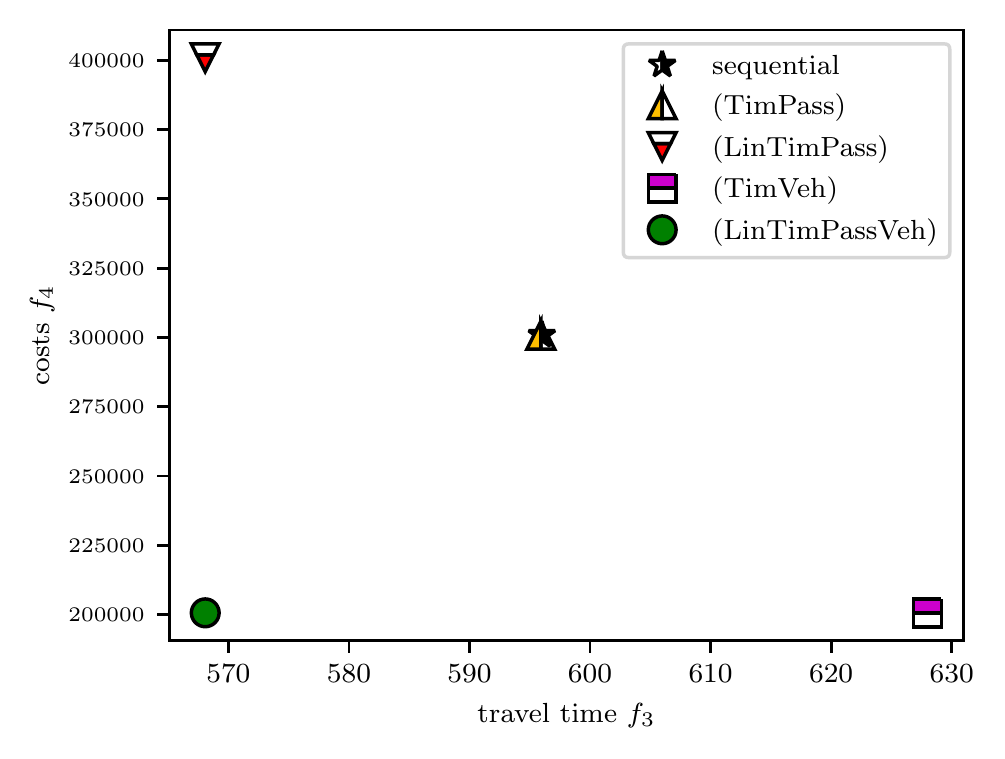}
	}
	\end{center}
	\caption{Travel time $f_\TT$ and costs $f_\cost$.}\label{fig:small:pareto}
\end{subfigure}
\caption{Evaluation of the price of sequentiality as well as of travel time $f_\TT$ and costs $f_\cost$ for data set \texttt{small} for parameters $\lambda_{\TT}=1000, \lambda_{\cost}=1$.}\label{fig:small:ptp}
\end{center}
\end{figure}

\begin{figure}[h]
\begin{center}
\begin{subfigure}{0.45\textwidth}
	\begin{center}
	\resizebox{\textwidth}{!}{
		\includegraphics{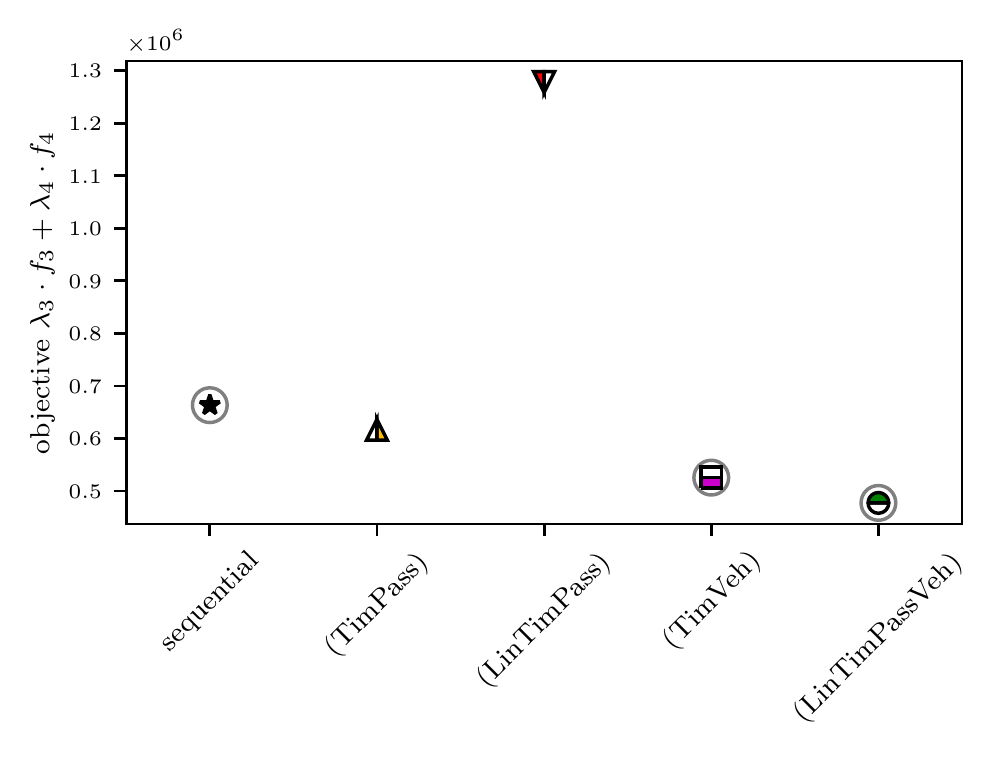}
	}
	\end{center}
	\vspace*{-5ex}
	\caption{$\lambda_{\TT}=10, \lambda_{\cost}=1$}\label{fig:toy1:weighted_sum}
\end{subfigure}
\begin{subfigure}{0.45\textwidth}
	\begin{center}
	\resizebox{\textwidth}{!}{
		\includegraphics{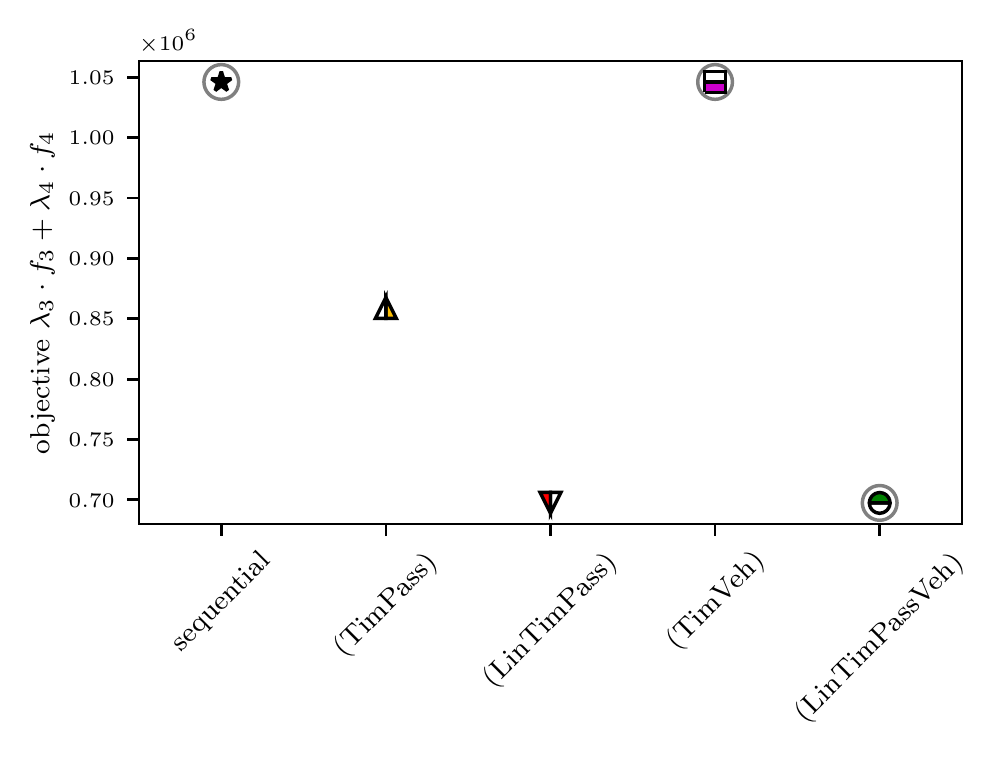}
	}
	\end{center}
	\vspace*{-5ex}
	\caption{$\lambda_{\TT}=1, \lambda_{\cost}=0$}\label{fig:toy2:weighted_sum}
\end{subfigure}
\begin{subfigure}{0.45\textwidth}
	\begin{center}
	\resizebox{\textwidth}{!}{
		\includegraphics{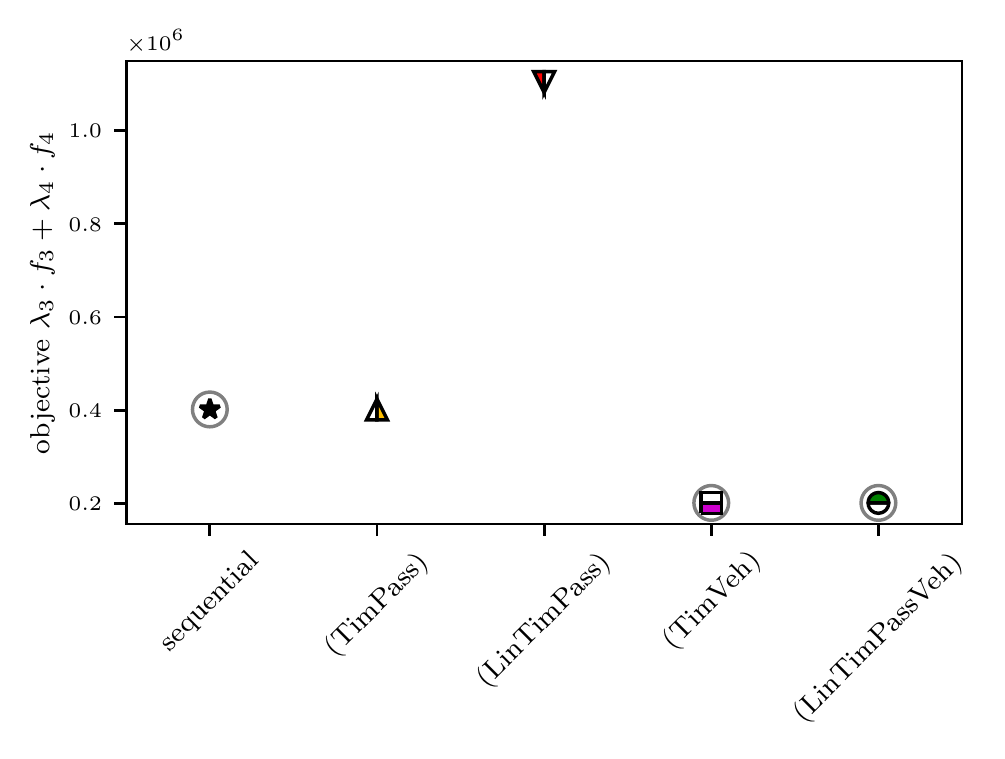}
	}
	\end{center}
	\vspace*{-5ex}
	\caption{$\lambda_{\TT}=0, \lambda_{\cost}=1$}\label{fig:toy3:weighted_sum}
\end{subfigure}
\begin{subfigure}{0.45\textwidth}
	\begin{center}
	\resizebox{\textwidth}{!}{
		\includegraphics{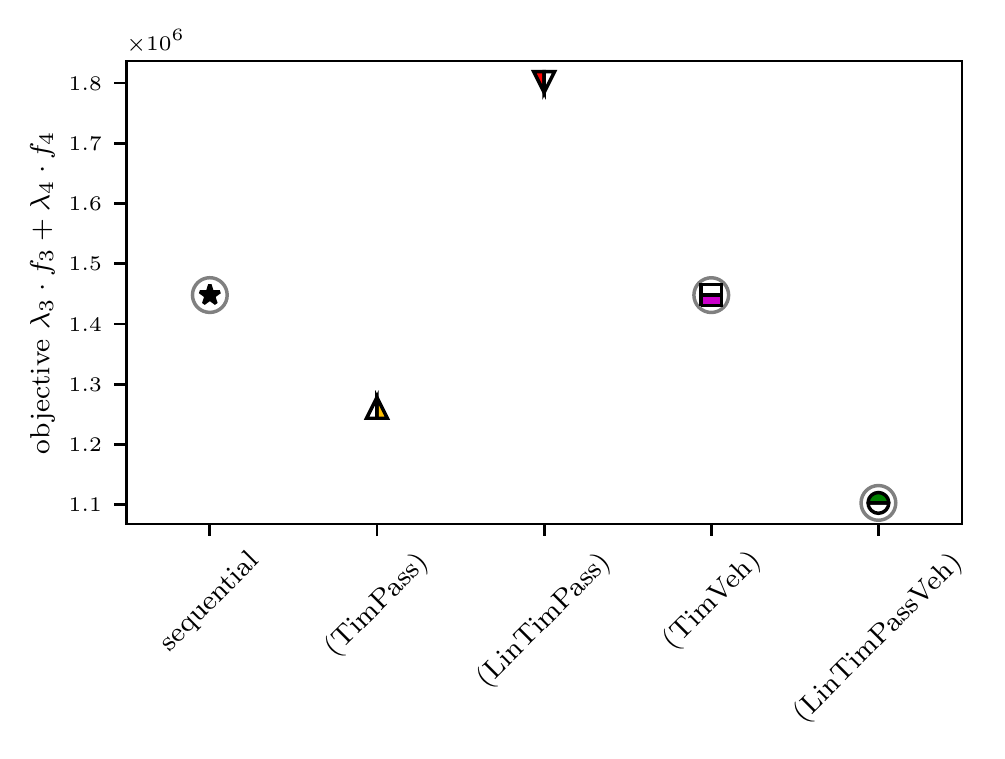}
	}
	\end{center}
	\vspace*{-5ex}
	\caption{$\lambda_{\TT}=40, \lambda_{\cost}=1$}\label{fig:toy4:weighted_sum}
\end{subfigure}
\caption{Evaluation of the objective $ \lambda_\TT \cdot f_\TT + \lambda_\cost \cdot f_\cost$ for data set \texttt{toy}. We marked the solution approaches for which an improvement can be guaranteed with a gray circle.}\label{fig:toy:ptp}
\end{center}
\end{figure}

\paragraph{Adding auxiliary objectives to the partially integrated solution approaches} Auxiliary objectives in earlier stages of the sequential process and in partially integrated solution approaches may be very helpful. Here, we especially see that using the line costs $f_\linecost$ as an approximation of the costs $f_\cost$ significantly improves the solution quality of \lintimpass{}. As detailed in Table~\ref{tab:parameters}, we changed the objective function for \lintimpass{} from $f_\TT$ to 
\[\lambda_\linecost \cdot f_\linecost + \lambda_\TT \cdot f_\TT\]
for various values of $\lambda_\linecost, \lambda_\TT$. Figure~\ref{fig:toy:pareto_restricted} shows that this enables \lintimpass{} to find solutions with low costs $f_\cost$ in addition to ones with low travel time $f_\TT$.

\begin{figure}
	\begin{center}
		\resizebox{0.7\textwidth}{!}{
		\includegraphics{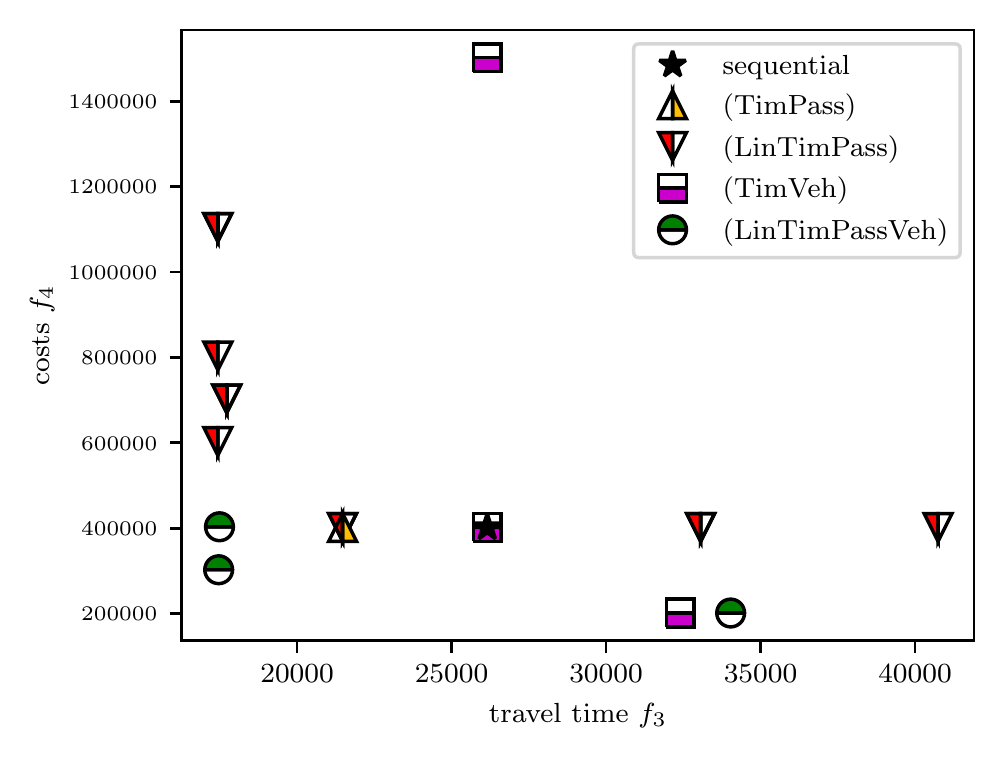}}
	\end{center}
	\vspace*{-5ex}
	\caption{Evaluation of travel time $f_\TT$ and costs $f_\cost$ for different parameter sets $(\lambda_\linecost,  \lambda_\TT,\lambda_\cost)$ for data set \texttt{toy}, see Table~\ref{tab:parameters}. Note that the solutions for \lintimveh{} have not been computed to optimality.}\label{fig:toy:pareto_restricted}
\end{figure}

\paragraph{Using varying weighted sum scalarization for finding Pareto solutions}
As discussed in Section~\ref{sec-notation}, there exists two possible interpretations of the objective function 
\[\lambda_\TT \cdot f_\TT + \lambda_\cost \cdot f_\cost.\]
On the one hand, $\frac{\lambda_\TT}{\lambda_\cost}$ can be regarded as the (known) value of time, such that the objective minimizes the generalized costs. In this case, it suffices to compute solutions for one set of parameters $\lambda_\TT, \lambda_\cost$.

On the other hand, we can consider both the travel time $f_\TT$ and the costs $f_\cost$ separately and use multi-objective optimization to find Pareto solutions, i.e., solutions that cannot be improved in both $f_\TT$ and $f_\cost$. 
Here, the parameter $\lambda_\TT, \lambda_\cost$ are scalarization parameters in a weighted sum approach. 
We consider this in Figures~\ref{fig:small:pareto} and \ref{fig:toy:pareto_restricted}.
First of all, note that when \lintimveh{} can be computed to optimality, all solutions are weakly Pareto optimal. 
For data set \texttt{small}, we can even show that the solution for $\lambda_\TT =1000, \lambda_\cost=1$ is an ideal solution. i.e., neither the travel time $f_\TT$ nor the cost $f_\cost$ can be improved.  Therefore, we do not need to consider further scalarization parameters for data set \texttt{small}.

For data set \texttt{toy}, the solver time of one hour did not suffice to compute \lintimveh{} to optimality. 
Therefore, one solution of \lintimveh{} is dominated by a solution of \timveh.  
In regard of the high computation time for \lintimveh{}, it is especially interesting to compare the solution quality of the partially integrated solution approaches. As expected, \timpass{} and \lintimpass{} can be used to  find solutions with low travel time $f_\TT$ while \timveh{} finds the solution with the lowest costs $f_\cost$. While the sequential solution approach finds a solution with comparatively low costs $f_\cost$, it is dominated both by solutions of the integrated approach \lintimveh{} and the partially integrated approaches \timpass{} and \lintimpass.   
Overall, Figure~\ref{fig:toy:pareto_restricted} shows that varying the scalarization parameters lead to interesting compromises between travel time $f_\TT$ and costs $f_\cost$ and that partially integrated solution approaches yield are a good way to generate varying solutions when the expense to solve the integrated problem is too high.

\section{Conclusions and further research}
\label{sec-conclusion}
In this paper, we consider the integration of sequential problems into a multi-stage problem and compare the solution quality for both approaches by considering the price of sequentiality. 
While in general the price of sequentiality is unbounded, we show that it highly depends on the definition of the sequential problems and the order in which these are solved.  
Applying these findings to the sequential problems line planning, timetabling, passenger routing and vehicle scheduling in public transport planning shows that partial integration leads to better solutions than the traditional sequential solution approach while keeping the problem size manageable.

On the one hand, we aim to extend the integrated model for public transportation.
For example, the integrated formulation can be extended to include the time slice model introduced in \cite{GGNS} in order to distribute the favored departure times of the passengers. Another important aspect, especially in high-demand networks is the integration of vehicle capacity for routing as in \cite{burggraeve2017integrating,goerigk2017line,PaeSchiSch18,Goverde19}. Both approaches are
omitted here to not further complicate the model.
Also, the relation to the eigenmodel presented in \cite{Sch16} is under consideration:
The different sequential solution approaches of the eigenmodel can be used to retrieve an (equivalent) multi-stage problem for line planning, timetabling, passenger routing and vehicle scheduling that can be compared to the one presented here.

On the other hand, we are working on extending the theoretical background of integrating sequential processes. One important aspect is to incorporate different possibilities to derive objective functions for the multi-stage problem. Here, we are especially considering multi-criteria objective functions, e.g.~in the context of complex systems, see \cite{DIETZ2020581}. 

\section*{Acknowledgments}
We thank an anonymous referee for strengthening Theorem~\ref{thm:bound-pos-single-objective} of this paper.

\clearpage

\bibliography{literature}
\bibliographystyle{alpha}

\clearpage

\appendix

\section{IP models for sequential public transport problems}
Here we provide details which have been left out in Section~\ref{sec:seq_sol}.

\subsection{The event-activity network (EAN)}
\label{app-EAN}
A formal definition of the event-activity network is given. We start with
the basic definition needed for timetabling and add the requirements for passenger routing
in Notation~\ref{nota-EAN-routing}.

\begin{notation}
  \label{nota-EAN-formal}
  Give a line plan $\cL$, the corresponding \emph{event-activity network}
  $\cN(\cL)=(\cE(\cL), \cA(\cL))$ is given by
  arrival and departure events of every line at every stop
  $\cE(\cL) = \cE_{\arr}(\cL) \cup \cE_{\dep}(\cL)$ with
  \begin{align*}
   \cE_{\arr}(\cL) &= \{(v, \arr, l) \colon v \in l \cap  V,l \in \cL\},\\
   \cE_{\dep}(\cL) &=  \{(v, \dep, l) \colon v \in l \cap V,l \in \cL\}
 \end{align*} 
 and drive, wait and transfer activities $\cA(\cL)= \cA_{\drive}(\cL) \cup \cA_{\wait}(\cL)
 \cup \cA_{\trans}(\cL)$ with
 \begin{align*}
	\cA_{\drive}(\cL) &= \{((v_1, \dep, l),(v_2, \arr, l)) \colon e=(v_1,v_2) \in l \cap E, l \in \cL\},\\
        \cA_{\wait}(\cL) &= \{((v, \arr, l),(v, \dep, l)) \colon v \in l \cap V, l \in \cL\},\\
	\cA_{\trans}(\cL) &= \{((v, \arr, l_1),(v, \dep, l_2)) \colon v \in l_1 \cap l_2 \cap V, l_1,l_2 \in \cL, l_1 \neq l_2\}.
 \end{align*}
 
For a simpler notation,
 we collect the drive and wait activities of line $l$ in $\cA(l)$ and the transfer activities between line $l_1$ and $l_2$ in $\cA(l_1,l_2)$.
\end{notation}

Finally, for each activity $a \in \cA(\cL)$, upper and lower bounds on the duration of the activity are given as $L_a,U_a \in \N$ with $0 \leq L_a \leq U_a$.
We define these lower and upper bounds of the activities according to the edges of the
underlying PTN.
More precisely, we assume that in the PTN lower and upper bounds $L_e^{\drive}\leq U_e^{\drive}$ for travel along edge $e \in E$ are given and that $L_v^{\wait} \leq U_v^{\wait}$, $L_v^{\trans}\leq U_v^{\trans}$, $v \in V$, are the stop dependent bounds for the dwell times of a vehicle and for the transfer time of a passenger. 
We set:
\begin{alignat}{5}
		L_a &= L_e^{\drive}, & U_a &= U_e^{\drive}, &&\quad a = ((u,l,\dep),(v,l,\arr))\in \cA_\drive(\cL), e = (u,v) \nonumber \\
		L_a &= L_v^{\wait}, & U_a &= U_v^{\wait}, &&\quad a = ((v,l,\arr),(v,l,\dep))\in \cA_\wait(\cL)   \label{explain-bounds}\\
		L_a &= L_v^{\trans}, & U_a &= U_v^{\trans}, &&\quad a = ((v,l_1,\arr),(v,l_2,\dep))\in \cA_\trans(\cL) \nonumber 
	\end{alignat}

To correctly model the passenger routing, the EAN has to be extended. Therefore, we include auxiliary events
representing source and target nodes for all OD pairs as well as activities connecting these
auxiliary events to the rest of the EAN. Along the lines of \cite{mariediss}, we receive:

\begin{notation}
\label{nota-EAN-routing}  
Given an event-activity network $\cN(\cL)=(\cE(\cL),\cA(\cL))$, the \emph{extended event-activity network}
$\bar{\cN}(\cL)=(\bar{\cE}(\cL), \bar{\cA}(\cL))$ is the following extension of $\cN(\cL)$: $\bar{\cE}(\cL)=\cE(\cL)\cup \cE_{\aux}(\cL)$, $\bar{\cA}(\cL) = \cA(\cL) \cup \cA(\cL)_{\aux}$ with
\begin{align*}
	\cE_{\aux}(\cL)&=\{(v,\src), (v,\tar) \colon v \in V\}\\
	\cA_{\aux}(\cL) &= \{((v,\src),(v,l,\dep)) \colon v \in l \cap V, l \in \cL\}\\
	& \qquad \cup \{((v,l,\arr),(v,\tar)) \colon v \in l \cap V, l \in \cL\}.
\end{align*}
For the auxiliary activities $a \in \cA_\aux(\cL)$ we set $L_a=U_a=0$.
\end{notation}

\subsection{Passenger routing (Pass)} \label{sec:app:passenger_routing}

Using the node-arc-incidence matrix $A(\cL)$ of $\bar{\cN}(\cL)$, binary
variables $p_a^{u,v}$ for determining whether activity $a \in \cA(\cL)$ is part of a $u - v$ path and vector $b^{u,v}$
with  \[b^{u,v}_{i} =
		\begin{cases}
			1, &\textup{if }  i=(u,\src)\\
			-1, &\textup{if } i=(v,\tar)\\
			0, &\textup{otherwise}
		\end{cases}		\]
		
we get the following flow formulation for finding a passenger routing.

\begin{alignat*}{5}
&& \min \sum_{(u,v) \in \OD} \sum_{a \in \cA(\cL)} & p_a^{u,v} \cdot L_a    && &\\
&&\st \quad A(\cL) \cdot (p^{u,v}) & =  b^{u,v} &&\qquad  (u,v) \in \OD \\
&& p_a^{u,v} & \in \{0,1\} && \qquad a \in \bar{\cA}(\cL), (u,v) \in \OD
\end{alignat*}

\subsection{Timetabling (Tim)}\label{sec:app:timetabling} 

The classical integer programming model for periodic timetabling uses
the so-called modulo parameters $z_a \in \Z$ for all $a \in \cA(\cL)$ to model
the periodicity of the timetable. As input, one needs the event-activity network $\cN(\cL)$
(which depends on stage~1) and 
the number of passengers
$w_a(\bar{p})$ which use activity $a \in \cA(\cL)$. These numbers depend on the paths
computed in stage~2.

\begin{alignat}{5}
	&&\min \sum_{a \in \cA(\cL)} w_a(\bar{p}) \cdot (\pi_j & - \pi_i + z_a  \cdot T) 
	&& \nonumber\\
	&& \st \quad  \pi_j - \pi_i + z_a \cdot T & \geq  L_a \qquad && a=(i,j) \in \cA(\cL) \label{PESP}\\
		&& \pi_j - \pi_i + z_a \cdot T & \leq  U_a &&  a=(i,j) \in \cA(\cL) \nonumber\\
		&&\pi_i & \in \{0, \ldots, T-1\} \quad && i \in \cE(\cL) \nonumber\\
		&& z_a & \in \Z && a \in \cA(\cL) \nonumber
\end{alignat}

Linearization of \eqref{TT3}:
\begin{alignat}{5}
&& \eta_a &= \bar{y}_{l_1} \cdot \bar{y}_{l_2} && a \in \cA(l_1,l_2)\nonumber \\
\intertext{can be linearized to}
  && \eta_a &\leq \bar{y}_{l_1} & \ \ \ &  a \in \cA(l_1,l_2) \label{tmp1}\\
  && \eta_a &\leq \bar{y}_{l_2} &&  a \in \cA(l_1,l_2) \label{tmp2}\\
  && \eta_a &\geq \bar{y}_{l_1} + \bar{y}_{l_2} -1 &&  a \in \cA(l_1,l_2). \label{tmp3}
\end{alignat}

\subsection{Vehicle scheduling (Veh)}\label{sec:app:veh} 
A \emph{trip} $\tau=(\rep, l)$ is determined by the line $l \in \cL$ that is covered and the period $\rep \in \cP$ in which the trip starts. It has to be operated by one vehicle end-to-end.
We use  boolean variables $x_{(\rep_1,l_1),(\rep_2,l_2)}$ to indicate if trip $(\rep_2,l_2)$ is operated by the vehicle that directly before that operated trip $(\rep_1,l_1)$ and boolean variables $x_{\depot,(\rep,l)}$, $x_{(\rep,l),\depot}$ to indicate if trip $(\rep,l)$ is the first or last trip in a vehicle route operated from or to the depot. If $x_{(\rep_1,l_1),(\rep_2,l_2)}=1$, the difference between the end time $\End_{\rep_1,l_1}$ of trip $(\rep_1,l_1)$ and the start time $\Start_{\rep_2,l_2}$ of trip $(\rep_2,l_2)$ has to be at least $L_{l_1,l_2}$.
We get the following IP formulation for the vehicle scheduling problem. 

\begin{alignat}{5}
		&& \min \Cost(\Veh)\qquad\qquad\qquad \qquad &&& \nonumber\\
		&& \st \quad  x_{(\rep_1,l_1),(\rep_2,l_2)}\cdot L_{l_1,l_2} & \leq \Start_{\rep_2,l_2}-\End_{\rep_1,l_1} \quad  &&\rep_1,\rep_2\in \cP,l_1,l_2\in \cL \nonumber\\		
		&& \sum_{\rep_1 \in \cP} \sum_{l_1 \in \cL} x_{(\rep_1,l_1),(\rep_2,l_2)} + x_{\depot,(\rep_2,l_2)} &= 1  \quad &&  \rep_2\in \cP,l_2\in \cL \label{VEHSCHED}\\
		&& \sum_{\rep_2\in \cP} \sum_{l_2 \in \cL} x_{(\rep_1,l_1),(\rep_2,l_2)} + x_{(\rep_1,l_1),\depot}  &= 1 \quad &&  \rep_1\in \cP,l_1\in \cL \nonumber \\
		&&  x_{(\rep_1,l_1),(\rep_2,l_2)} & \in \{0,1\} && \rep_1, \rep_2 \in \cP, l_1,l_2 \in \cL \nonumber\\
		&& x_{\depot, (\rep,l)}, x_{(\rep,l),\depot} &\in \{0,1\} && p\in \cP, l \in \cL \nonumber
\end{alignat}

Depending on stages 1 to 3, the vehicle scheduling problem can be formulated as follows. Note that while the start and end times $\Start_{\rep,l}$, $\End_{\rep,l}$ of trips  $(\rep,l)$, $\rep \in \cP, l \in \cL^0$, as well as the duration of a trip on line $l$ $\dur_l$ are now variables, their values can be predetermined as they only depend on fixed variables $\pi,z$.

The objective function $f_4(\bar{y}, \bar{p}, \bar{\pi}, \bar{z}; x)$ represents the costs $\Cost(\Veh)$ of vehicle schedule $\Veh$. They depend on the duration and distance covered by the vehicles as well as the number of vehicles needed. We distinguish between the time and distance needed for trips, i.e., when passengers may be on board, and the time and distance needed for so-called empty (or deadheading) trips when no passengers may be on board. These empty trips include time between two consecutive trips in a vehicle route, the distance of potential relocation (or deadheading) trips as well as trips to and from the depot at the end and the start of a vehicle route, respectively. The details are listed in Table~\ref{tab:obj:vs}.

\begin{flalign}
\min \quad & f_4(\bar{y}, \bar{p}, \bar{\pi}, \bar{z}; x) = && \nonumber\\ 
		&   \gamma_1 \cdot \sum_{l \in \cL^0} \bar{y}_l \cdot \dur_l\label{obj:vs:line_dur}&&\\
		&  +\gamma_2 \cdot \sum_{l \in \cL^0}\bar{y}_l \cdot \length_l \label{obj:vs:line_length} &&\\
		& + \gamma_3 \cdot \sum_{\rep_1 \in \cP} \sum_{l_1 \in \cL^0} 
			\Big( 
				\sum_{\rep_2 \in\cP} \sum_{l_2 \in \cL^0} 
					\big( \Start_{(\rep_2,l_2)}-\End_{(\rep_1,l_1)}\big)\cdot x_{(\rep_1,l_2),(\rep_2,l_2)}\nonumber&&\\
				&\qquad \qquad\qquad+ x_{\depot, (\rep_1,l_1)}\cdot L_{\depot,l_1} + x_{(\rep_1,l_1),\depot}\cdot L_{l_1,\depot}
			\Big)\label{obj:vs:turn_dur} &&\\
		&  + \gamma_4 \cdot \sum_{\rep_1 \in \cP} \sum_{l_1 \in \cL^0} 
			\Big( 
				\sum_{\rep_2 \in \cP} \sum_{l_2 \in \cL^0}
				\big(
					x_{(\rep_1,l_1),(\rep_2,l_2)} \cdot D_{l_1,l_2}
				\big) \nonumber &&\\
				&\qquad \qquad\qquad+ x_{\dep, (\rep_1,l_1)}\cdot D_{\dep,l_1}
				+ x_{(\rep_1,l_1),\dep}\cdot D_{l_1,\dep}
			\Big) \label{obj:vs:turn_dist} &&\\
			& + \gamma_5 \cdot \sum_{\rep \in \cP}\sum_{l \in \cL^0} x_{\depot,(\rep,l)}\label{obj:vs:veh} &&
	\end{flalign}
	
	\begin{alignat}{5}
		\st && \dur_l&=\sum_{a=(i,j) \in \cA(l)} (\bar{\pi}_j-\bar{\pi}_i +\bar{z}_a \cdot T) \;\; 
			&&  l \in\cL^0 \label{TV1}\\
		&& \Start_{\rep,l}&=\rep\cdot T+\bar{\pi}_{\first(l)} && \rep \in \cP, l \in \cL^0  \label{TV2}\\
		&& \End_{\rep,l}&=\rep\cdot T+\bar{\pi}_{\first(l)}+\dur_l && \rep \in \cP, l \in \cL^0 \label{TV3}\\
		&& \Start_{\rep_2,l_2}-\End_{\rep_1,l_1} &\geq  x_{(\rep_1,l_1),(\rep_2,l_2)}\cdot L_{l_1,l_2}
			&&\nonumber\\
			&& &  \quad-M'\cdot (1-x_{(\rep_1,l_1),(\rep_2,l_2)}) \quad&&  \rep_1,\rep_2\in \cP,l_1,l_2 \in \cL^0 \label{V1}\\
			&& \bar{y}_{l_2} &=  \sum_{\rep_1 \in \cP} \sum_{l_1 \in \cL^0} x_{(\rep_1,l_1),(\rep_2,l_2)} &&\nonumber\\
			&& &  \quad+ x_{\depot,(\rep_2,l_2)}   &&  \rep_2\in \cP,l_2\in \cL^0\label{LV1}\\
		&& \bar{y}_{l_1} &=  \sum_{\rep_2\in \cP} \sum_{l_2 \in \cL^0} x_{(\rep_1,l_1),(\rep_2,l_2)}&&\nonumber\\
			&& &  \quad+ x_{(\rep_1,l_1),\depot}  &&  \rep_1\in \cP,l_1\in \cL^0\label{LV2}\\
		&& x_{(\rep,l),\bullet} & \leq  \bar{y}_l  && \rep \in \cP, l \in \cL^0 \label{LV3}\\
		&& x_{\bullet, (\rep,l)} & \leq  \bar{y}_l  && \rep \in \cP, l \in \cL^0 \label{LV4}
	\end{alignat}
		\begin{alignat}{5}
\dur_l, \Start_{\rep,l}, \End_{\rep,l}  
&\in \N, \label{var1}\\ 
x_{(\rep_1,l_1),(\rep_2,l_2)},  x_{\depot,(\rep,l)}, x_{(\rep,l),\depot} &\in \{0,1\} \label{var2}
	\end{alignat}

\begin{table}[h]
{\small
\begin{center}
\begin{tabular}{cccl}
\toprule
\textbf{Cost type} &  & \makecell{\textbf{Weighting}\\\textbf{factor}} &  \textbf{Components}\\ 
\midrule
\makecell{time based\\ trip costs} & \eqref{obj:vs:line_dur} & $\gamma_1$ & $\dur_l$ duration of a trip on line $l$ \\ 
\midrule
\makecell{distance based\\ trip costs} & \eqref{obj:vs:line_length} & $\gamma_2$ & $\length_l$ length of a trip on line $l$  \\ 
\midrule
\makecell{time based\\ empty trip costs} & \eqref{obj:vs:turn_dur} & $\gamma_3$ & \makecell[l]{$\Start_{(\rep_2,l_2)}-\End_{(\rep_1,l_1)}$ duration of empty trip\\
$L_{\depot,l}$ time from depot to start of line $l$\\
$L_{l,\depot}$ time from end of line $l$ to depot}\\ 
\midrule 
\makecell{distance based\\ empty trip costs} & \eqref{obj:vs:turn_dist} & $\gamma_4$ & \makecell[l]{$D_{l_1,l_2}$ length of empty trip\\
$D_{\depot,l}$ distance from depot to start of line $l$\\
$D_{l,\depot}$ distance from end of line $l$ to depot} \\ 
\midrule 
\makecell{vehicle based\\ costs} & \eqref{obj:vs:veh} & $\gamma_5$ & number of vehicles \\ 
\bottomrule
\end{tabular} 
\end{center}
\caption{All parts of the objective function considered in vehicle scheduling.}\label{tab:obj:vs}
}
\end{table}

In constraints \eqref{TV1}, \eqref{TV2} and \eqref{TV3}, the duration of a trip on line $l$ as well as the start and end time of trip $(\rep,l)$, $\rep \in \cP, l \in \cL^0$, are computed. Note that the variables $\bar{\pi},\bar{z}$ are already fixed and that $\first(l)$, $l \in \cL^0$, represents the first event in line $l$. Constraint \eqref{V1} ensures that the time between the start and end of consecutive trips is sufficiently long while constraints \eqref{LV1} to \eqref{LV4} model the vehicle flow for all operated lines. 

\end{document}